\documentclass[11pt]{article}
\usepackage{amsmath, amssymb,amscd,comment,appendix,lpic,multicol}
\usepackage[mathscr]{eucal}
\usepackage{amscd}
\usepackage{amsthm}
\usepackage{thmtools}
\usepackage{cite}
\usepackage{stmaryrd}
\usepackage{enumerate}
\usepackage{url}
\usepackage{macros}
\usepackage{thms}

\usepackage[margin=1in]{geometry}

\def\Dist{\operatorname{Dist}}
\def\Int{\operatorname{Int}}

\newcommand\subsetsim{\mathrel{%
  \ooalign{\raise0.3ex\hbox{\scalebox{0.9}{$\subset$}}\cr\hidewidth\raise-0.8ex\hbox{\scalebox{0.9}{$\sim$}}\hidewidth\cr}}}

\begin{document}

\title{Distance structures for generalized metric spaces}

\date{October 13, 2016}

\author{
Gabriel Conant\\
University of Notre Dame\\
gconant@nd.edu
}

\maketitle

\begin{abstract}
Let $\cR=(R,\p,\leq,0)$ be an algebraic structure, where $\p$ is a commutative binary operation with identity $0$, and $\leq$ is  a translation-invariant total order with least element $0$. Given a distinguished subset $S\seq R$, we define the natural notion of a ``generalized" $\cR$-metric space, with distances in $S$. We study such metric spaces as first-order structures in a relational language consisting of a distance inequality for each element of $S$. We first construct an ordered additive structure $\cS^*$ on the space of quantifier-free $2$-types consistent with the axioms of $\cR$-metric spaces with distances in $S$, and show that, if $A$ is an $\cR$-metric space with distances in $S$, then any model of $\Th(A)$ logically inherits a canonical $\cS^*$-metric. Our primary application of this framework concerns countable, universal, and homogeneous metric spaces, obtained as generalizations of the rational Urysohn space. We adapt previous work of Delhomm\'{e}, Laflamme, Pouzet, and Sauer to fully characterize the existence of such spaces. We then fix a countable totally ordered commutative monoid $\cR$, with least element $0$, and consider $\cU_\cR$, the countable Urysohn space over $\cR$. We show that quantifier elimination for $\Th(\cU_\cR)$ is characterized by continuity of addition in $\cR^*$, which can be expressed as a first-order sentence of $\cR$ in the language of ordered monoids. Finally, we analyze an example of Casanovas and Wagner in this context. 
\end{abstract}

\numberwithin{figure}{section}

The fundamental objects of interest in this paper are metric spaces. Specifically, we study the behavior of metric spaces as combinatorial structures in relational languages. This is the setting of a vast body of literature (e.g. \cite{CaWa}, \cite{DLPS}, \cite{vThebook}, \cite{Sol}, \cite{TeZibdd}, \cite{TeZiSIR}) focusing on topological dynamics of automorphism groups and Ramsey properties of countable homogeneous structures. Our goal is to develop the model theory of metric spaces in this setting. We face the immediate obstacle that the notion of ``metric space" is not very well controlled by classical first-order logic, in the sense that models of the theory of a metric space need not be metric spaces. Indeed, this is a major motivation for working in continuous logic and model theory for \textit{metric structures}, which are always complete metric spaces with the metric built into the logic (see \cite{BBHU}). However, we wish to study the model theory of (possibly incomplete) metric spaces treated as combinatorial structures (specifically, labeled graphs where complexity is governed by the triangle inequality). In some sense, we will sacrifice the global topological structure of metric spaces for the sake of understanding local combinatorial complexity. We will also develop an algebraic structure on distances sets of metric spaces, as a means to analyze this combinatorial complexity. 

Another benefit of our framework is that it will be flexible enough to encompass generalized metric spaces with distances in arbitrary ordered additive structures. This setting appears often in the literature, with an obvious example of extracting a metric from a valuation. Other examples include \cite{Nar}, where Narens considers topological spaces ``metrizable" by a generalized metric over an ordered abelian group, as well as \cite{MoSh}, where Morgan and Shalen use metric spaces over ordered abelian groups to generalize the notion of an $\R$-tree.  Also, in \cite{CaWa}, Casanovas and Wagner use the phenomenon of ``infinitesimal distance" to construct a theory without the strict order property that does not eliminate hyperimaginaries. We will analyze this example at the end of Section \ref{sec:ex}.

We will consider metric spaces as first-order \textit{relational} structures. However, when working outside of this first-order setting, it will usually be much more convenient to think of metric spaces as ``sorted" structures consisting of a set of points together with a distance function into a set of distances. Distinguishing between these two viewpoints will be especially important, and so we will very carefully explain the precise first-order relational setting in which we will be working. This explanation requires the following basic definitions. 

\begin{definition}\label{def:DM}
Let $\LDS=\{\p,\leq,0\}$ be the \textbf{language of ordered monoids} consisting of a binary function symbol $\p$, a binary relation symbol $\leq$, and a constant symbol $0$. Fix an $\LDS$-structure $\cR=(R,\p,\leq,0)$. 
\begin{enumerate}
\item $\cR$ is a \textbf{distance magma} if
\begin{enumerate}[$(i)$]
\item (\textit{totality}) $\leq$ is a total order on $R$;
\item (\textit{positivity}) $r\leq r\p s$ for all $r,s\in R$;
\item (\textit{order}) for all $r,s,t,u\in R$, if $r\leq t$ and $s\leq u$ then $r\p s\leq t\p u$;
\item (\textit{commutativity}) $r\p s=s\p r$ for all $r,s\in R$;
\item (\textit{unity}) $r\p 0=r=0\p r$ for all $r\in R$.
\end{enumerate}
\item $\cR$ is a \textbf{distance monoid} if it is a distance magma and
\begin{enumerate}[\hspace{-5pt}$(vi)$]
\item (\textit{associativity}) $(r\p s)\p t=r\p(s\p t)$ for all $r,s,t\in R$.
\end{enumerate}
\end{enumerate}
\end{definition}

Note that if $\cR$ is a distance magma, then it follows from the \emph{positivity} and \emph{unity} axioms that $0$ is the least element of $R$. Moreover, given $r,s,t\in R$ if $r\leq s$ then $r\p t\leq s\p t$ by the \emph{order} axiom. However, it is worth emphasizing that this translation-invariance is not strict: we may have $r<s$, while $r\p t=s\p t$. In particular, a distance magma may be finite, in which case if $s\in R$ is the maximal element then $r\p s=s$ for all $r\in R$. See Example \ref{introEx} below.

\begin{remark}Recall that, according to \cite{Bour}, a \textit{magma} is simply a set together with a binary operation. After consulting standard literature on ordered algebraic structures (e.g. \cite{Cliff}), one might refer to a distance magma as a \textit{totally and positively ordered commutative unital magma}, and a distance monoid as a \textit{totally and positively ordered commutative monoid}. So our terminology is partly chosen for the sake of brevity. We are separating the associativity axiom because it is not required for our initial results and, more importantly, associativity will frequently characterize some useful combinatorial property of metric spaces (see Proposition \ref{compProps}$(e)$, Proposition \ref{4VCassoc}, Exercise \ref{freeAmal}). 
\end{remark}

Next, we observe that the notion of a distance magma allows for a reasonable definition of a generalized metric space. Definitions of a similar flavor can be found in \cite{AlTr}, \cite{MoSh},  and \cite{Nar}. 

\begin{definition}\label{def:MS}
Suppose $\cR=(R,\p,\leq,0)$ is a distance magma. Fix a nonempty set $A$ and a function $d:A\times A\func R$. We call $(A,d)$ an \textbf{$\cR$-colored space}, and define the \textbf{distance set of $(A,d)$}, denoted $\Dist(A,d)$, to be the image of $d$ in $R$. Given an $\cR$-colored space $(A,d)$, we say $d$ is an \textbf{$\cR$-metric on $A$} if
\begin{enumerate}[\hspace{10pt}$(i)$]
\item for all $x,y\in A$, $d(x,y)=0$ if and only if $x=y$;
\item for all $x,y\in A$, $d(x,y)=d(y,x)$;
\item for all $x,y,z\in A$, $d(x,z)\leq d(x,y)\p d(y,z)$.
\end{enumerate}
In this case, $(A,d)$ is an \textbf{$\cR$-metric space}. A \textbf{generalized metric space} is an $\cR$-metric space for some distance magma $\cR$.
\end{definition}

We now detail the first-order setting of this paper. Given a distance magma $\cR$, we first define relational languages suitable for studying $\cR$-metric spaces with distances in some distinguished subset of $R$. In particular, given $S\seq R$, with $0\in S$, we define a first-order language $\cL_S=\{d(x,y)\leq s:s\in S\}$, where, for each $s\in S$, $d(x,y)\leq s$ is a binary relation symbol in the variables $x$ and $y$.

For later purposes, we describe the interpretation of arbitrary \emph{$\cR$-colored} spaces as $\cL_S$-structures. Fix a distance magma $\cR=(R,\p,\leq,0)$ and a subset $S\seq R$, with $0\in S$. Given an $\cR$-colored space $\cA=(A,d_A)$, we interpret $\cA$ as an $\cL_S$-structure by interpreting the symbol $d(x,y)\leq s$ as $\{(a,b)\in A^2:d_A(a,b)\leq s\}$. We let $\Th_{\cL_S}(\cA)$ denote the complete $\cL_S$-theory of the resulting $\cL_S$-structure. \textbf{All model theoretic statements and results about generalized metric spaces will be in this relational context.}

Recall that we have another language, namely, $\LDS$. Many of the results in this paper and its sequel \cite{CoDM2} associate model theoretic properties of generalized metric spaces with algebraic and combinatorial properties of distance magmas, which can often be expressed in a first-order way using $\LDS$. Therefore, the reader should consider $\LDS$ as an auxiliary language used mostly for convenience. Altogether, it is worth emphasizing again that throughout this paper, we will be working with two different classes of structures. The primary class is that of generalized metric spaces, and our main goal is to develop the model theory of these objects in the relational setting discussed above. The secondary class of structures is the class of distance magmas. We will not focus on this class from a model theoretic perspective.

One motivation for the study of generalized distance structures comes from the wide variety of examples this notion encompasses. The following are a few examples arising naturally in the literature.

\begin{example}\label{introEx}$~$
\begin{enumerate}
\item If $\cR=(\R^{\geq0},+,\leq,0)$ then $\cR$-metric spaces coincide with usual metric spaces. In this case, we refer to $\cR$-metric spaces as \textit{classical metric spaces}.
\item If $\cR=(\R^{\geq0},\max,\leq,0)$ then $\cR$-metric spaces coincide with classical ultrametric spaces.
\item Given $S\seq\R^{\geq0}$, with $0\in S$, we consider classical metric spaces with distances restricted to $S$. This is the context of \cite{DLPS}, which has inspired much of the following work (especially Section \ref{sec:4VC}). If $S$ satisfies the property that, for all $r,s\in S$, the subset $\{x\in S:x\leq r+s\}$ contains a maximal element, then we have a distance magma $\cS=(S,+_S,\leq,0)$, where $r+_S s:=\max\{x\in S:x\leq r+s\}$. In this case, $\cS$-metric spaces are precisely classical metric spaces with distances restricted to $S$. This situation is closely studied by Sauer in \cite{Sa13} and \cite{Sa13b}. In Section \ref{sec:closure}, we develop this example in full generality.
\end{enumerate}
\end{example}

A more important motivation for considering distance structures and metric spaces at this level of generality is that we will obtain a class of structures invariant under elementary equivalence. Roughly speaking, we will show that any model of the $\cL_S$-theory of an $\cR$-metric space with distances in $S$ is itself a generalized metric space over a canonical distance magma $\cS^*$, which depends only on $S$ and $\cR$, but often contains distances not in $S$. For example, suppose $\cA=(A,d_A)$ is a classical metric space over $(\R^{\geq0},+,\leq,0)$, which contains points of arbitrarily small distances. Then we can use compactness to build models of the $\cL_{\Q^{\geq0}}$-theory of $\cA$, which contain distinct points infinitesimally close together. Therefore, when analyzing these models, we must relax the notion of distance and consider a ``nonstandard" extension of the distance set. The first main result of this paper is that such an extension can always be found.

\begin{alphatheorem}\label{prethm}
Let $\cR$ be a distance magma. For any $S\seq R$, with $0\in S$, there is an $\LDS$-structure $\cS^*=(S^*,\cps,\leq^*,0)$ satisfying the following properties.
\begin{enumerate}[$(a)$]
\item $\cS^*$ is a distance magma.
\item $(S^*,\leq^*)$ is an extension of $(S,\leq)$, and $S$ is dense in $S^*$ (with respect to the order topology).
\item Given $r,s\in S$, if $r\p s\in S$ then $r\cps s=r\p s$.
\item Suppose $\cA$ is an $\cR$-metric space, with $\Dist(\cA)\seq S$. Fix $M\models \Th_{\cL_S}(\cA)$.
\begin{enumerate}[\hspace{-5pt}$(i)$]
\item For all $a,b\in M$, there is a unique $\alpha=\alpha(a,b)\in S^*$ such that, given any $s\in S$, we have $M\models d(a,b)\leq s$ if and only if $\alpha\sleq s$.
\item If $d_M:M\times M\func S^*$ is defined such that $d_M(a,b)=\alpha(a,b)$, then $(M,d_M)$ is an $\cS^*$-metric space.
\end{enumerate}
\end{enumerate}
\end{alphatheorem}

The object $\cS^*$ from Theorem \ref{prethm} is obtained by defining a distance magma structure on the space of quantifier-free $2$-types consistent with a natural set of axioms for $\cR$-metric spaces with distances in $S$. We will also give explicit combinatorial descriptions of the set $S^*$ and the operation $\cps$. Moreover, we will isolate conditions under which, in part $(d)$ of this theorem, the requirement $\Dist(\cA)\seq S$ can be weakened (for example, in order to keep $\cL_S$ countable). Theorem \ref{prethm} appears again in its final form as Theorem \ref{thm:models}.

We then consider the existence of an \textit{$\cR$-Urysohn space over $S$}, denoted $\URS{S}{\cR}$, where $S$ is a countable subset of some distance magma $\cR$. When it exists, $\URS{S}{\cR}$ is a countable, ultrahomogeneous $\cR$-metric space with distance set $S$, which is universal for finite $\cR$-metric spaces with distances in $S$. In \cite{DLPS}, Delhomm\'{e}, Laflamme, Pouzet, and Sauer characterize the existence of $\URS{S}{\cR}$ in the case of $\cR=(\R^{\geq0},+,\leq,0)$. In Section \ref{sec:4VC} we show that after appropriate translation, the same characterization goes through for any $\cR$. A corollary is that given a \emph{countable} distance magma $\cR=(R,\p,\leq,0)$, the $\cR$-Urysohn space $\cU_\cR:=\URS{R}{\cR}$ exists if and only if $\p$ is associative. Therefore, in Section \ref{sec:QE} we fix a countable distance \emph{monoid} $\cR$ and consider $\TUS{\cR}:=\Th_{\cL_R}(\cU_\cR)$, the first-order $\cL_R$-theory of $\cU_\cR$. Our second main result characterizes quantifier elimination for $\TUS{\cR}$ in terms of continuity in $\cR^*=(R^*,\cp_R,\sleq,0)$. 

\begin{alphatheorem}\label{prethm2}
If $\cR$ is a countable distance monoid then $\TUS{\cR}$ has quantifier elimination if and only if, for all $s\in R$, the function $x\mapsto x\cp_R s$, from $R^*$ to $R^*$, is continuous with respect to the order topology on $R^*$.
\end{alphatheorem}

This theorem appears again as Theorem \ref{QE}. A corollary of this result is the existence of an $\LDS$-sentence $\vphi$ such that, if $\cR$ is a countable distance monoid, then $\TUS{\cR}$ has quantifier elimination if and only if $\cR\models\vphi$. In Section \ref{sec:ex}, we will see that quantifier elimination holds for most natural examples found in previous literature, including the rational Urysohn space and sphere. It is interesting to note that a demonstration of quantifier elimination (in a discrete relational language) for these classical examples does not appear explicitly in previous literature. The closest related result is \cite{CaWa}, in which Casanovas and Wagner demonstrate quantifier elimination for a certain theory, which they also show does not have the strict order property or elimination of hyperimaginaries. A consequence of the framework developed in this paper is that their theory is precisely that of the rational Urysohn sphere (see Section \ref{sec:ex}).  

Note that if $\cR$ is finite then quantifier elimination for $\Th(\cU_\cR)$ follows from standard results on $\aleph_0$-categorical \Fraisse\ limits. However, if $\cR$ is infinite then $\Th(\cU_\cR)$ is not $\aleph_0$-categorical. In this situation, quantifier elimination for \Fraisse\ limits can fail (see Example \ref{noQE}). Therefore, Theorem \ref{prethm2} uncovers a class of possibly non-$\aleph_0$-categorical \Fraisse\ limits in which quantifier elimination  is characterized by natural analytic behavior of the structure. 

The characterization of quantifier elimination for $\Th(\cU_\cR)$ also initiates a program of study concerning the relationship between model theoretic properties of $\cU_\cR$ and algebraic properties of $\cR$. This is the subject of the sequel to this paper \cite{CoDM2}. The result is a rich class of first-order structures without the strict order property, which represent a wide range of complexity in examples both classical and exotic (e.g. stable theories of refining equivalence relations as ultrametric Urysohn spaces; the simple, unstable random graph as the Urysohn space over $\{0,1,2\}$; and the rational Urysohn space, which has the strong order property). Moreover, these measures of complexity are characterized in \cite{CoDM2} by natural algebraic and combinatorial properties of the monoid $\cR$.

We give a brief summary of the paper. In Section \ref{sec:FOS}, we axiomatize generalized metric spaces. In Sections \ref{sec:S*} and \ref{sec:FOL}, we construct $\cS^*$ and prove Theorem \ref{prethm}. Section \ref{sec:closure} develops specific analytic features of $\cS^*$, which will be needed for later results. In Section \ref{sec:4VC} we characterize the existence of generalized Urysohn spaces; and we prove Theorem \ref{prethm2} in Section \ref{sec:QE}. Section \ref{sec:ex} contains examples.

\section{Axioms for Generalized Metric Spaces}\label{sec:FOS}

Our first main goal is to construct the structure $(S^*,\cps,\leq^*,0)$ described in Theorem \ref{prethm}, where $S$ is some subset of a distance magma $\cR=(R,\p,\leq,0)$. Each step of the construction is motivated by an attempt to capture the first-order theory of $\cR$-metric spaces, and so we first formulate axioms for these structures.

\begin{definition}\label{def:TMS}
Suppose $\cR=(R,\p,\leq,0)$ is a distance magma. Fix $S\seq R$, with $0\in S$.
\begin{enumerate}
\item Recall that we defined the first-order language $\cL_S=\{d(x,y)\leq s:s\in S\}$ where, for each $s\in S$, $d(x,y)\leq s$ is a binary relation symbol in the variables $x$ and $y$. Let $d(x,y)>s$ denote the negation $\neg (d(x,y)\leq s)$.
\item Let $\TMS{S}{\cR}$ denote the union of the following schemes of $\cL_S$-sentences:
\begin{enumerate}[\hspace{10pt}(MS1)]
\item $\forall x\forall y(d(x,y)\leq 0\leftrightarrow x=y)$;
\item for all $s\in S$,
$$
\forall x\forall y(d(x,y)\leq s\leftrightarrow d(y,x)\leq s);
$$
\item for all $r,s,t\in S$ such that there is no $x\in S$ with $t<x\leq r\p s$,
$$
\forall x\forall y\forall z((d(x,y)\leq r\wedge d(y,z)\leq s)\rightarrow d(x,z)\leq t);
$$
\item if $S$ has a maximal element $s$,
$$
\forall x\forall y ~d(x,y)\leq s.
$$
\end{enumerate}
\end{enumerate}
\end{definition}

\begin{remark}\label{rem:cutaxiom}
Let $\cR$ be a distance magma, and fix $S\seq R$ with $0\in S$. From (MS1) and (MS3) we deduce that for all $r,s\in S$, if  $r\leq s$ then
$$
\TMS{S}{\cR}\models \forall x\forall y(d(x,y)\leq r\rightarrow d(x,y)\leq s).
$$
In particular, for all $s\in S$, $\TMS{S}{\cR}\models \forall x~d(x,x)\leq s$. It follows that there is a unique quantifier-free $1$-type (with no parameters) consistent with $\TMS{S}{\cR}$. 
\end{remark}

It is not difficult to see that $\cR$-metric spaces, with distances in $S$, satisfy the axioms in $\TMS{S}{\cR}$.  However, it will be helpful in later work to know when $\cR$-metric spaces, with distances possibly outside of $S$, still satisfy $\TMS{S}{\cR}$. Toward this end, we first define a notion of approximation, which captures the extent to which atomic $\cL_S$-formulas can distinguish distances in $R$.

\begin{definition}\label{def:Sapprox} Suppose $\cR=(R,\p,\leq,0)$ is a distance magma. Fix $S\seq R$, with $0\in S$.
\begin{enumerate}
\item Define
$$
\Int(S,\cR)=\{\{0\}\}\cup\{(r,s]:r,s\in S,~r<s\},
$$
where, given $r,s\in S$ with $r<s$, $(r,s]$ denotes the interval $\{x\in R:r<x\leq s\}$. (These sets are chosen to reflect quantifier-free $\cL_S$-formulas of the form $``r<d(x,y)\leq s":=(d(x,y)>r)\wedge (d(x,y)\leq s)$.)
\item Given $X\seq R$, a function $\Phi:X\func \Int(S,\cR)$ is an \textbf{$(S,\cR)$-approximation of $X$} if $x\in\Phi(x)$ for all $x\in X$. When $\Phi(x)\neq\{0\}$, we use the notation $\Phi(x)=(\Phi^-(x),\Phi^+(x)]$.
\item Suppose $(x_1,\ldots,x_n)\in R^n$ and $\Phi$ is an $(S,\cR)$-approximation of $\{x_1,\ldots,x_n\}$. Let $N_\Phi(x_1,\ldots,x_n)$ denote $\Phi(x_1)\times\ldots\times\Phi(x_n)\seq R^n$. \end{enumerate}
\end{definition}

Note that if $\Phi$ is an $(S,\cR)$-approximation of $X\seq R$, and $0\in X$, then we must have $\Phi(0)=\{0\}$. Therefore, whenever defining a specific $(S,\cR)$-approximation $\Phi$, we will always tacitly define $\Phi(0)=\{0\}$, and let $\Phi^+(0)=0$.

Next, we define a condition on $\cR$-metric spaces $\cA$ and sets $S\seq R$, which will ensure $\cA\models \TMS{S}{\cR}$.

\begin{definition}\label{def:approx}
Suppose $\cR=(R,\p,\leq,0)$ is a distance magma.
\begin{enumerate}
\item A triple $(r,s,t)\in R^3$ is an \textbf{$\cR$-triangle} if $r\leq s\p t$, $s\leq r\p t$, and $t\leq r\p s$. 

\item Given $S\seq R$, let $\Delta(S,\cR)$ denote the set of $\cR$-triangles in $S^3$.

\item Given an $\cR$-metric space $\cA=(A,d_A)$, let $\Delta(\cA)$ denote the set of $\cR$-triangles of the form $(d_A(a,b),d_A(b,c),d_A(a,c))$ for some $a,b,c\in A$.  

\item Fix an $\cR$-metric space $\cA$ and a subset $S\seq R$. We write $\Delta(\cA)\subsetsim\Delta(S,\cR)$ if, for all $(r,s,t)\in\Delta(\cA)$,
\begin{enumerate}[$(i)$]
\item there is an $(S,\cR)$-approximation of $\{r,s,t\}$, and
\item $N_\Phi(r,s,t)\cap\Delta(S,\cR)\neq\emptyset$ for any $(S,\cR)$-approximation $\Phi$ of $\{r,s,t\}$.
\end{enumerate}
\end{enumerate}
\end{definition}

\begin{remark}Fix a distance magma $\cR$, an $\cR$-metric space $\cA$, and $S\seq R$.
\begin{enumerate}
\item Definition \ref{def:approx}$(4i)$ holds for all $(r,s,t)\in\Delta(\cA)$ if and only if $\Dist(\cA)$ is contained in the convex closure of $S$ in $R$.  
\item $\Dist(\cA)\seq S$ if and only if $\Delta(\cA)\seq\Delta(S,\cR)$.
\item If $\Dist(\cA)\seq S$ then $\Delta(\cA)\subsetsim\Delta(S,\cR)$.
\end{enumerate}
\end{remark}

\newpage

\begin{example}\label{ex:metDense}
\item Let $\cR=(\R^{\geq0},+,\leq,0)$ and fix a classical metric space $\cA$. Then $\Delta(\cA)\subsetsim\Delta(\N,\cR)$ and $\Delta(\cA)\subsetsim\Delta(\Q^{\geq0},\cR)$. If $\Dist(\cA)\seq[0,1]$ then $\Delta(\cA)\subsetsim\Delta(\{0,\frac{1}{n},\frac{2}{n},\ldots,1\},\cR)$ for any $n>0$.
\end{example}

\begin{proposition}\label{convAx}
Let $\cR=(R,\p,\leq,0)$ be a distance magma and fix an $\cR$-metric space $\cA$. If $S\seq R$ and $\Delta(\cA)\subsetsim\Delta(S,\cR)$, then $\cA\models\TMS{S}{\cR}$.
\end{proposition}
\begin{proof}
The axiom schemes (MS1) and (MS2) are immediate; and (MS4) follows by Definition \ref{def:approx}$(4i)$. So it remains to verify axiom scheme (MS3). Fix $r,s,t\in S$ such that there is no $x\in S$ with $t<x\leq r\p s$. Suppose $a,b,c\in A$, with $\cA\models d(a,b)\leq r\wedge d(b,c)\leq s$. Let $d_A(a,b)=u$, $d_A(b,c)=v$, and $d_A(a,c)=w$. Then we have $u\leq r$ and $v\leq s$, and we want to show $w\leq t$. Suppose, toward a contradiction, that $t<w$. Using \ref{def:approx}$(4i)$, we may define an $(S,\cR)$-approximation $\Phi$ of $\{u,v,w\}$ such that $\Phi^+(u)=r$, $\Phi^+(v)=s$, and $\Phi^-(w)=t$. By \ref{def:approx}$(4ii)$, there is some $(r',s',t')\in N_\Phi(u,v,w)\cap \Delta(S,\cR)$. Then $t<t'\leq r'\p s'\leq r\p s$, which contradicts the choice of $r,s,t\in S$.
\end{proof}

Suppose $\cR$ is a distance magma and $S\seq R$, with $0\in S$. The distance magma $\cS^*$ from Theorem \ref{prethm} will have the property that any $\cS^*$-metric space satisfies $\TMS{S}{\cR}$ and, moreover, any $\caL_S$-structure satisfying $\TMS{S}{\cR}$ can be equipped with an $\cS^*$-metric in a coherent and canonical way. In other words, $\TMS{S}{\cR}$ axiomatizes the class of $\cS^*$-metric spaces (see Proposition \ref{axModels} for the precise statement). Once $\cS^*$ has been defined and analyzed, this result will follow quite easily. The work lies in the construction of $\cS^*$, and the proof that $\cS^*$ is a distance magma.

\section{Construction of $\cS^*$}\label{sec:S*} 

\textbf{Throughout all of Section \ref{sec:S*}, we fix a distance magma $\cR=(R,\p,\leq,0)$, and work with a fixed subset $S\seq R$, with $0\in S$.} The goal of this section is to construct $\cS^*$ satisfying Theorem \ref{prethm}. The essential idea is that we are defining a distance magma structure on the space of quantifier-free $2$-types consistent with $\TMS{S}{\cR}$. This statement is made precise by Proposition \ref{prop:2types} and Definition \ref{def:cps}.

\subsection{Construction of $(S^*,\leq^*)$}

The order $(S^*,\sleq)$ will be the smallest complete linear order containing $(S,\leq)$, in which every non-maximal element of $S$ has an immediate successor. In particular, $(S^*,\sleq)$ will depend only on $(S,\leq)$, and not on the ambient distance magma $\cR$.  We will show that $S^*$ is in bijective correspondence with the space of quantifier-free $2$-types consistent with $\TMS{S}{\cR}$ (see Proposition \ref{prop:2types}).

\begin{definition}$~$
\begin{enumerate}
\item A subset $X\seq S$ is a \textbf{cut in $S$} if it is closed upward and, if $S$ has a maximal element $s$, then $s\in X$.
\item Define $S^*$ to be the set of cuts in $S$. Define $\sleq$ on $S^*$ such that, given $X,Y\in S^*$, $X\sleq Y$ if and only if $Y\seq X$.
\end{enumerate}
\end{definition}

The linear order $(S^*,\sleq)$ is a slight variation on the \emph{Dedekind-MacNeille completion} of $(S,\leq)$ (see \cite[Section 11]{MacN}). For instance, we allow $\emptyset$ to be a cut in $S$ in the case that $S$ has no maximal element. Moreover, the standard Dedekind-MacNeille completion only involves cuts which either have no infimum in $S$ or contain their infimum. This motivates the following terminology. 

\begin{definition}
Suppose $X$ is a cut in $S$. Then $X$ is a \textbf{gap cut} if it has no infimum in $S$; $X$ is a \textbf{successor cut} if $\inf X$ exists in $S$ and $\inf X\not\in X$; and $X$ is a \textbf{principal cut} if $\inf X$ exists in $S$ and $\inf X\in X$.
\end{definition}

Note that $X\seq S$ is a principal cut if and only if $X=S^{\geq r}$ for some $r\in S$, and $X$ is a successor cut if and only if $X=S^{>r}$ for some non-maximal $r\in S$ with no immediate successor in $S$. In particular, we identify $S$ with the set of principal cuts (as a subset of $S^*$) via the injective map $r\mapsto S^{\geq r}$. We will also replace the elements of $S^*\backslash S$ with new symbols suggestive of their behavior.

\begin{notation}\label{S*expl}
For the rest of the paper, we use the following description of $(S^*,\sleq)$. We identify $S^*$ with 
$$
S\cup \{r^+:r\in S,~S^{>r}\text{ is a successor cut}\}\cup \{g_X:X\seq S\text{ is a gap cut}\},
$$
where $r^+$ and $g_X$ are distinct new symbols not in $S$. Then $(S^*,\sleq)$ is completely determined by $(S,\leq)$ and the following rules (see Figure \ref{Sfig}):
\begin{enumerate}
\item If $S^{>r}$ is a successor cut then $r\sle r^+\sle s$ for all $s\in S^{>r}$.
\item If $X\seq S$ is a gap cut then $r\sle g_X \sle s$ for all $r\in S\backslash X$ and $s\in X$.
\end{enumerate}
\begin{figure}[h]
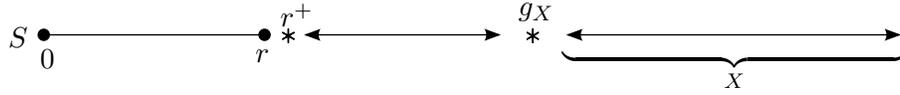

\begin{center}
\begin{lpic}[t(0mm),b(0mm),l(2mm),draft,clean]{line(.65)}
\lbl[t]{-4,3;$S$}
\lbl[t]{2,-1.5;$0$}
\lbl[t]{46,-1.5;$r$}
\lbl[t]{53,8;$r^+$}
\lbl[t]{102,8;$g_X$}
\lbl[t]{142.5,-1.5;$\underbrace{\makebox[1.8in]{}}_X$}
\end{lpic}
\end{center}
\caption{New elements of $S^*$}
\label{Sfig}
\end{figure}
Note that $S^*$ has a maximal element, which occurs in one of two ways:
\begin{enumerate}[$(i)$]
\item If $S$ has a maximal element $s$, then this is also the maximal element of $S^*$.
\item If $S$ has no maximal element then $\emptyset$ is a gap cut in $S$, and so $g_\emptyset$ is the maximal element of $S^*$. 
\end{enumerate}
We will use $\omega_S$ to denote the maximal element of $S^*$. We can distinguish between the two cases above by declaring either $\omega_S\in S$ or $\omega_S\not\in S$.
\end{notation}

\begin{proposition}
$(S^*,\sleq)$ is a complete linear order.
\end{proposition}
\begin{proof}
This is a tedious exercise, which we leave to the reader. Using Notation \ref{S*expl}, if $S_1=S\cup\{r^+:S^{>r}\text{ is a successor cut}\}$, then $(S^*,\sleq)$ is the Dedekind-MacNeille completion of $(S_1,\sleq)$.
\end{proof}

In light of the last result, we may calculate infima and suprema in $S^*$. By convention, when considering $\emptyset$ as a subset of $S^*$, we let $\inf\emptyset=\omega_S$ and $\sup\emptyset= 0$. We will work with the \emph{order topology} on $S^*$ given by sub-basic open intervals of the form $[0,\alpha)$ or $(\alpha,\omega_S]$ for some $\alpha\in S^*$.  

\begin{proposition}\label{dense}$~$
\begin{enumerate}[$(a)$]
\item For all $\alpha,\beta\in S^*$, if $\alpha\sle\beta$ then there is some $t\in S$ such that $\alpha\sleq t\sle\beta$.
\item If $X\seq S^*$ is nonempty and $\inf X\in S$, then $\inf X\in X$.
\end{enumerate}
\end{proposition}
\begin{proof}
 Part $(a)$. Fix $\alpha,\beta\in S^*$ with $\alpha\sle\beta$. We may clearly assume $\alpha\not\in S$. We consider the case that $\alpha=r^+$ for some successor cut $S^{>r}$. The case when $\alpha=g_X$ for some gap cut $X\seq S$ is similar and left to the reader. If $\beta=s\in S$ or $\beta=s^+$ for some $s\in S$, then $r<s$ and so, by assumption on $r$, there is some $t\in S$ such that $r<t<s$. On the other hand, if $\beta=g_X$ for some gap cut $X\seq S$, then $r\not\in X$ and so there is $t\not\in X$, with $r<t$. In any case, $\alpha\sle t\sle\beta$.
 
Part $(b)$. This follows from the fact that any non-maximal $r\in S$ has an immediate successor in $S^*$, namely either $r^+$ or an immediate successor in $S$.
\end{proof}

Part $(a)$ of the previous result will be used frequently throughout the entirety of the paper. Therefore, for smoother exposition, we will say \textbf{by density of $S$} when using this fact.

Finally, we connect $(S^*,\sleq)$ back to the first-order setting.

\begin{definition}
Given $\alpha\in S^*$, define the set of $\cL_S$-formulas
\[
p_\alpha(x,y)=\{d(x,y)\leq s:s\in S,~\alpha\sleq s\}\cup \{d(x,y)>s:s\in S,~s\sle\alpha\}.
\]
\end{definition}

\begin{proposition}\label{prop:2types}
Let $S^{\textnormal{qf}}_2(\TMS{S}{\cR})$ denote the space of complete quantifier-free $\cL_S$-types $p(x,y)$ in two variables, such that $p(x,y)\cup \TMS{S}{\cR}$ is consistent. Then the map $\alpha\mapsto p_\alpha(x,y)$ is a bijection from $S^*$ to $S^{\textnormal{qf}}_2(\TMS{S}{\cR})$.
\end{proposition}
\begin{proof}
We first show that the map is well-defined. Fix $\alpha\in S^*$, and let $\Sigma=p_\alpha(x,y)\cup \TMS{S}{\cR}$. If $\Sigma$ is consistent then, by axiom schemes (MS1) and (MS2), $p_\alpha(x,y)$ uniquely determines an element of $S^{\textnormal{qf}}_2(\TMS{S}{\cR})$. So it suffices to show $\Sigma$ is consistent. If $\alpha=0$ then the $\cR$-metric space with one element satisfies $\Sigma$. Assume $\alpha\neq 0$. Define an $\cL_S$-structure $A=\{a,b\}$ such that, given $s\in S$, $d(x,y)\leq s$ is symmetric, reflexive, and holds on $(a,b)$ if and only if $\alpha\sleq s$. Note that if $A\models\TMS{S}{\cR}$ then $A\models\Sigma$. So we show $A\models\TMS{S}{\cR}$. Axioms (MS1), (MS2), and (MS4) are clear. For (MS3), fix $r,s,t\in S$ such that there is no $x\in S$ with $t<x\leq r\p s$. In particular, $\max\{r,s\}\leq t$. Fix $x,y,z\in A$ such that $A\models d(x,y)\leq r\wedge d(y,z)\leq s$. If $x=z$ then $A\models d(x,z)\leq t$. If $x\neq z$ then either $y=x$ or $y=z$. Since $\max\{r,s\}\leq t$, we have $A\models d(x,z)\leq t$ in either case.

For injectivity, fix $\alpha,\beta\in S^*$, with $\alpha\sle\beta$. By density of $S$, there is $s\in S$ such that $\alpha\sleq s\sle\beta$. Then $p_\alpha(x,y)$ contains $d(x,y)\leq s$ and $p_\beta(x,y)$ contains $d(x,y)>s$.

Finally, we show surjectivity. Given $p(x,y)\in S^{\textnormal{qf}}_2(\TMS{S}{\cR})$, let $X(p)$ be the set of $s\in S$ such that $p$ contains $d(x,y)\leq s$. By axiom schemes (MS1) and (MS2), $p$ is completely determined by $X(p)$. So it suffices to fix $p(x,y)\in S^{\textnormal{qf}}_2(\TMS{S}{\cR})$ and show there is some $\alpha\in S^*$ with $X(p)=X(p_\alpha)$. Let $X=X(p)$, and note that $X$ is a cut by Remark \ref{rem:cutaxiom} and axiom (MS4). If $X=S^{\geq r}$ for some $r\in S$ then $X=X(p_r)$; if $X$ is a successor cut of the form $S^{>r}$ then $X=X(p_{r^+})$; and if $X$ is a gap cut then $X=X(p_{g_X})$. 
\end{proof}

\subsection{Construction of $\cps$}

The definition of $\cps$ is motivated by the trivial observation that, given  $r,s\in R$, 
$$
r\p s=\sup\{t\in R:(r,s,t)\text{ is an $\cR$-triangle}\}.
$$
Given $\alpha,\beta \in S^*$, we define $\alpha\cps\beta$ in an analogous way. 

\begin{definition}\label{def:cps}$~$
\begin{enumerate} 
\item Fix $\alpha,\beta\in S^*$.
\begin{enumerate}[$(a)$]
\item Given $\gamma\in S^*$, the triple $(\alpha,\beta,\gamma)$ is a \textbf{logical $S^*$-triangle} if 
$$
\TMS{S}{\cR}\cup p_\alpha(x,y)\cup p_\beta(y,z)\cup p_\gamma(x,z)
$$
is consistent.
\item Define $\Sigma(\alpha,\beta)=\{\gamma\in S^*:\text{$(\alpha,\beta,\gamma)$ is a logical $S^*$-triangle}\}$.
\item Define $\alpha\cps\beta=\sup \Sigma(\alpha,\beta)$.
\end{enumerate}
\item Let $\cS^*$ denote the $\LDS$-structure $(S^*,\cps,\sleq,0)$.
\end{enumerate}
\end{definition}

To prove that $\cS^*$ is a distance magma, we will need a more explicit expression for $\cps$. We start with some basic properties of logical $S^*$-triangles.

\begin{notation}\label{not:Sapprox}
The notions formulated in Definition \ref{def:Sapprox} only require the linear order $(R,\leq,0)$, and thus remain sensible with $\cS^*$ in place of $\cR$. We say \textbf{$S$-approximation} to mean $(S\cup\{\omega_S\},\cS^*)$-approximation, and we let $\Int^*(S)$ denote $\Int(S\cup\{\omega_S\},\cS^*)$.
\end{notation}

\begin{remark}\label{rem:Sapprox}In the next few proofs, we will tacitly use the following observations. Suppose $\Phi$ is an $S$-approximation of $X\seq S^*$ and $\alpha\in X$.
\begin{enumerate}
\item $\Phi^+(\alpha)\in S\cup\{\omega_S\}$, and if $\alpha\neq 0$ then $\Phi^-(\alpha)\in S$.
\item $\Phi(\alpha)\cap S\neq\emptyset$. 

This is clear if $\Phi^+(\alpha)\in S$. Otherwise $\Phi^+(\alpha)=\omega_S\not\in S$, in which case $S$ has no maximal element and so $\Phi(\alpha)$ contains $S^{>\Phi^-(\alpha)}\neq\emptyset$.
\end{enumerate}
\end{remark}

\begin{proposition}\label{prop:approx}$~$
\begin{enumerate}[$(a)$]
\item Given $\alpha,\beta,\gamma\in S^*$, $\gamma\in\Sigma(\alpha,\beta)$ if and only if $N_\Phi(\alpha,\beta,\gamma)\cap\Delta(S,\cR)\neq\emptyset$ for every $S$-approximation $\Phi$ of $\{\alpha,\beta,\gamma\}$.
\item If $\alpha,\beta\in S^*$ then $\max\{\alpha,\beta\}\in\Sigma(\alpha,\beta)$, and so $\max\{\alpha,\beta\}\sleq\alpha\cps\beta$.
\end{enumerate}
\end{proposition}
\begin{proof}
Part $(a)$. The reverse direction follows by compactness and Proposition \ref{convAx}. For the forward direction, let $\Phi$ be an $S$-approximation of $\{\alpha,\beta,\gamma\}$. We may assume $\alpha\sleq\beta\sleq\gamma$. 

Suppose first that $\beta=\omega_S$ (which means $\gamma=\omega_S$ as well). Choose $r\in\Phi(\alpha)\cap S$. We may find $s\in\Phi(\omega_S)\cap S$, with $r\leq s$ (if $\omega_S\in  S$ then let $s=\omega_S$; and if $\omega_S\not\in S$ then choose any $s\in\Phi(\omega_S)\cap S$ with $r\leq s$). Then $(r,s,s)\in N_\Phi(\alpha,\beta,\gamma)\cap\Delta(S,\cR)$.

Now suppose $\beta\sle\omega_S$. We claim that there are $r,s\in S$ such that $\alpha\sleq r\sleq\Phi^+(\alpha)$, $\beta\sleq s\sleq\Phi^+(\beta)$, and $r\leq s\sleq\Phi^+(\gamma)$. Indeed, since $\beta\sle\omega_S$, we may use density of $S$ to fix $u\in S$ such that $\beta\sleq u$. Now let $s=\min\{u,\Phi^+(\beta),\Phi^+(\gamma)\}$ and $r=\min\{s,\Phi^+(\alpha)\}$. 

Since $\Phi^-(\gamma)\sle \gamma$, $\alpha\sleq r$, $\beta\sleq s$, and $(\alpha,\beta,\gamma)$ is a logical $S^*$-triangle, it follows from axiom scheme (MS3) that there is some $t\in S$ with $\Phi^-(\gamma)<t\leq r\p s$. After replacing $t$ by $\min\{t,\Phi^+(\gamma)\}$, we may assume $t\in\Phi(\gamma)$. We may also assume $s\leq t$ (if $t<s$ then $s\in \Phi(\gamma)\cap S$, and so we replace $t$ with $s$). Then $(r,s,t)\in N_\Phi(\alpha,\beta,\gamma)\cap\Delta(S,\cR)$. 

Part $(b)$. Fix $\alpha,\beta\in S^*$, with $\alpha\sleq\beta$. We use part $(a)$ to show $\beta\in\Sigma(\alpha,\beta)$. Let $\Phi$ be an $S$-approximation of $\{\alpha,\beta\}$. By density of $S$ we may assume that if $\Phi^+(\beta)\not\in S$ then $\beta=\omega_S\not\in S$. After replacing $\Phi^+(\alpha)$ with $\min\{\Phi^+(\alpha),\Phi^+(\beta)\}$, we may assume $\Phi^+(\alpha)\sleq\Phi^+(\beta)$. Choose $r,s\in S$ as follows. If $\Phi^+(\beta)\in S$ then let $r=\Phi^+(\alpha)$ and $s=\Phi^+(\beta)$. Otherwise, $\beta=\omega_S\not\in S$ and we choose $r\in\Phi(\alpha)\cap S$ arbitrarily and $s\in\Phi(\beta)\cap S$ with $r\leq s$. In either case, $(r,s,s)\in N_\Phi(\alpha,\beta,\gamma)\cap \Delta(S,\cR)$.
\end{proof}




The next result gives explicit expressions for $\cps$.

\begin{lemma}\label{lem:explSum}$~$
\begin{enumerate}[$(a)$]
\item If $\alpha,\beta\in S^*$ then
$$
\alpha\cps\beta=\inf\{\sup\{x\in S:x\leq r\p s\}:r,s\in S,~\alpha\sleq r,~\beta\sleq s\}.
$$
\item If $r,s\in S$ then $r\cps s=\sup\{x\in S:x\leq r\p s\}$.
\item If $\alpha,\beta\in S^*$ then $\alpha\cps \beta=\inf\{r\cps s:r,s\in S,~\alpha\sleq r,~\beta\sleq s\}$.
\item If $r,s,r\p s\in S$ then $r\cps s=r\p s$.
\end{enumerate}
\end{lemma}
\begin{proof}
Part $(a)$. For $r,s\in S$, let $X(r,s)=\{x\in S:x\leq r\p s\}$. Fix $\alpha,\beta\in S^*$, and let $\gamma=\inf\{\sup X(r,s):r,s\in S,~\alpha\sleq r,~\beta\sleq s\}$. We first show $\alpha\cps\beta\sleq\gamma$. Suppose not. Then there are $r,s\in S$, with $\alpha\sleq r$, $\beta\sleq s$, and $\sup X(r,s)\sle\alpha\cps\beta$. By density of $S$, we may fix $t\in S$ with $\sup X(r,s)\sleq t\sle\alpha\cps\beta$, and then, by definition of $\alpha\cps\beta$, fix $\delta\in\Sigma(\alpha,\beta)$ with $t\sle\delta$. Let $\Phi$ be an $S$-approximation of $\{\alpha,\beta,\delta\}$ such that $\Phi^+(\alpha)=r$, $\Phi^+(\beta)=s$, and $\Phi^-(\delta)=t$. Since $\delta\in\Sigma(\alpha,\beta)$, we may use Proposition \ref{prop:approx}$(a)$ to fix  $(u,v,w)\in N_\Phi(\alpha,\beta,\gamma)\cap\Delta(S,\cR)$. Then $w\leq u\p v\leq r\p s$, and so $w\sleq\sup X(r,s)\sleq t=\Phi^-(\delta)$, which is a contradiction. 

To finish the proof, we show $\gamma\in\Sigma(\alpha,\beta)$. By Proposition \ref{prop:approx}$(b)$, we may assume $\beta\sle \gamma$. Without loss of generality, we also assume $\alpha\sleq\beta$.  Let $\Phi$ be an $S$-approximation of $\{\alpha,\beta,\gamma\}$. We want to show $N_\Phi(\alpha,\beta,\gamma)\cap \Delta(S,\cR)\neq\emptyset$. We claim that there are $r,s\in S$ such that $\alpha\sleq r\sleq\Phi^+(\alpha)$, $\beta\sleq s\sleq\Phi^+(\beta)$ and $r\leq s\sle\gamma$. Indeed, since $\beta\sle\gamma$ we may use density of $S$ to fix $u\in S$ with $\beta\sleq u\sle\gamma$, and we then let $s=\min\{u,\Phi^+(\beta)\}$ and $r=\min\{s,\Phi^+(\alpha)\}$.

If $\Phi^+(\gamma)\sle\sup X(r,s)$ then $(r,s,\Phi^+(\gamma))\in N_\Phi(\alpha,\beta,\gamma)\cap\Delta(S,\cR)$. So suppose $\sup X(r,s)\sleq\Phi^+(\gamma)$. We have $\max\{s,\Phi^-(\gamma)\}\sle\gamma\sleq\sup X(r,s)$, and so there is some $t\in X(r,s)$ such that $\max\{s,\Phi^-(\gamma)\}<t$. Then $t\sleq\sup X(r,s)\sleq\Phi^+(\gamma)$ and so, altogether, $(r,s,t)\in N_\Phi(\alpha,\beta,\gamma)\cap\Delta(S,\cR)$. 

Part $(b)$ follows easily from part $(a)$. Part $(c)$ combines parts $(a)$ and $(b)$. Part $(d)$ is immediate from part $(b)$.  
\end{proof}

We can now prove the main goal.

\begin{theorem}\label{thm:S*DM}
Suppose $\cR=(R,\p,\leq,0)$ is  a distance magma and fix $S\seq R$, with $0\in S$. Then $\cS^*=(S^*,\cps,\sleq,0)$ is a distance magma.
\end{theorem} 
\begin{proof}
By construction, $(S^*,\sleq,0)$ is a linear order with least element 0, and $\cps$ is clearly commutative. By Propositions \ref{prop:2types} and \ref{prop:approx}$(b)$, $0$ is the identity element of $\cS^*$. Finally, fix $\alpha,\beta,\gamma,\delta\in S^*$, with $\alpha\sleq\gamma$ and $\beta\sleq\delta$. By the explicit expression for $\cps$ in Lemma \ref{lem:explSum}$(c)$, we have $\alpha\cps\beta\sleq\gamma\cps\delta$.
\end{proof}

The following consequence of the previous work will be useful later on.

\begin{proposition}\label{prop:expl}
Suppose $\cR=(R,\p,\leq,0)$ is a distance magma and $S\seq R$, with $0\in S$. Then $\cS^*$-triangles coincide with logical $S^*$-triangles.
\end{proposition}
\begin{proof}
Any logical $S^*$-triangle is an $\cS^*$-triangle by definition of $\cps$. Conversely, suppose $(\alpha,\beta,\gamma)$ is an $\cS^*$-triangle. Without loss of generality, we assume $\alpha\sleq\beta\sleq\gamma$. By the proof of Lemma \ref{lem:explSum}$(a)$, we may also assume $\gamma\sle\alpha\cps\beta$, and so there is $\delta\in\Sigma(\alpha,\beta)$ with $\gamma\sle\delta$.  We use Proposition \ref{prop:approx}$(a)$ to show $(\alpha,\beta,\gamma)$ is a logical $S^*$-triangle. Fix an $S$-approximation $\Phi$ of $\{\alpha,\beta,\gamma\}$. We want to show $N_\Phi(\alpha,\beta,\gamma)\cap\Delta(S,\cR)\neq\emptyset$. Without loss of generality, we may replace $\Phi^+(\alpha)$ by $\min\{\Phi^+(\alpha),\Phi^+(\gamma)\}$, and $\Phi^+(\beta)$ by $\min\{\Phi^+(\beta),\Phi^+(\gamma)\}$, and thus assume $\Phi^+(\alpha)\sleq\Phi^+(\gamma)$ and $\Phi^+(\beta)\sleq\Phi^+(\gamma)$. 

Since $\delta\in\Sigma(\alpha,\beta)$ and $\Phi^-(\gamma)\sle\gamma\sle\delta$, we may use Proposition \ref{prop:approx}$(a)$ to find $(r,s,u)\in\Delta(S,\cR)$ such that $r\in\Phi(\alpha)$, $s\in\Phi(\beta)$, and $\Phi^-(\gamma)<u$. Note that $\max\{r,s\}\leq\Phi^+(\gamma)$. Let $t=\min\{\Phi^+(\gamma),u\}$. Then $(r,s,t)\in N_\Phi(\alpha,\beta,\gamma)\cap\Delta(S,\cR)$. 
\end{proof}


We will eventually be interested in \emph{associative} distance magmas (i.e. distance monoids). When that time comes, it will be helpful to know that in order to check associativity of $\cps$, it suffices to just consider elements of $S$.

\begin{proposition}\label{associative*}
Suppose $\cR$ is a distance magma and $S\seq R$, with $0\in S$. If $r\cps(s\cps t)=(r\cps s)\cps t$ for all $r,s,t\in S$, then $\cps$ is associative on $S^*$.
\end{proposition}
\begin{proof}
To show $\cps$ is associative on $S^*$ it suffices, by commutativity of $\cps$, to fix $\alpha,\beta,\gamma\in S^*$ and show $(\alpha\cps\beta)\cps\gamma\sleq\alpha\cps(\beta\cps\gamma)$. To show this inequality it suffices by Lemma \ref{lem:explSum}$(c)$ to fix $r,u\in S$, with $\alpha\sleq r$ and $\beta\cps\gamma\sleq u$, and show $(\alpha\cps\beta)\cps\gamma\sleq r\cps u$. So fix such $r,u\in S$.

We claim that, since $\beta\cps\gamma\sleq u$, there are $s,t\in S$ such that $\beta\sleq s$, $\gamma\sleq t$ and $s\cps t\sleq u$. Indeed, if $\beta\cps\gamma\sle u$ then this is immediate from Lemma \ref{lem:explSum}$(c)$; and if $\beta\cps\gamma=u$ then, as $u\in S$, we obtain the desired $s,t\in S$ by Lemma \ref{lem:explSum}$(c)$ combined with Proposition \ref{dense}$(b)$. Now we have
\[
(\alpha\cps\beta)\cps\gamma\sleq (r\cps s)\cps t=r\cps(s\cps t)\sleq r\cps u.\qedhere
\]
\end{proof}

\begin{remark}\label{rem:TMS}
Given a distance magma $\cR$ and a subset $S\seq R$, we can now treat $S$ as a subset of the distance magma $\cS^*$, and define the $\cL_S$-theory $\TMS{S}{\cS^*}$ using Definition \ref{def:TMS}. So it is worth observing that $\TMS{S}{\cS^*}=\TMS{S}{\cR}$. In particular, (MS1), (MS2), and (MS4) are clearly the same in each case. Given $r,s,t\in S$, it follows from Lemma \ref{lem:explSum} that there is an $x\in S$ with $t<x\leq r\p s$ if and only if there is an $x\in S$ such that $t<x\sleq r\cps s$. So (MS3) is also the same in each case.
\end{remark}

\section{First-Order Theories of Metric Spaces}\label{sec:FOL}

In this section, we collect the previous results and prove Theorem \ref{prethm}. We first show that $\TMS{S}{\cR}$ can be thought of as a collection of axioms for the class of $\cS^*$-metric spaces (as a subclass of $\cS^*$-colored spaces).

\begin{definition} Suppose $\cR$ is a distance magma and $S\seq R$, with $0\in S$. Let $A$ be an arbitrary $\cL_S$-structure. Then $A$ is \textbf{$\cS^*$-colorable} if, for all $a,b\in A$, there is a (unique) $\alpha=\alpha(a,b)\in S^*$ such that $A\models p_\alpha(a,b)$. In this case, we define $d_A:A\times A\func S^*$ such that $d_A(a,b)=\alpha(a,b)$.
\end{definition}

\begin{proposition}\label{axModels}
Suppose $\cR$ is a distance magma and $S\seq R$, with $0\in S$. 
\begin{enumerate}[$(a)$]
\item Let $A$ be an $\cL_S$-structure. If $A\models\TMS{S}{\cR}$ then $A$ is $\cS^*$-colorable. 
\item Let $\cA=(A,d_A)$ be an $\cS^*$-colored space. Then $\cA\models\TMS{S}{\cR}$ if and only if $\cA$ is an $\cS^*$-metric space. 
\end{enumerate}
\end{proposition} 
\begin{proof}
Part $(a)$. By Proposition \ref{prop:2types}. 

Part $(b)$. If $\cA\models\TMS{S}{\cR}$ then $\cA$ is an $\cS^*$-metric space by axioms schemes (MS1) and (MS2), and the definition of $\cps$. Conversely, suppose $\cA$ is an $\cS^*$-metric space. Let $S_\omega=S\cup\{\omega_S\}$. Note that $\Delta(S,\cS^*)\seq\Delta(S_\omega,\cS^*)$ and, by Lemma \ref{lem:explSum}$(b)$, $\Delta(S,\cS^*)=\Delta(S,\cR)$. Combined with Propositions \ref{prop:approx}$(a)$ and \ref{prop:expl}, we conclude $\Delta(\cA)\subsetsim\Delta(S_\omega,\cS^*)$, and so $\cA\models \TMS{S_\omega}{\cS^*}$ by Proposition \ref{convAx}. It is easy to see that $\TMS{S}{\cS^*}\seq \TMS{S_\omega}{\cS^*}$ and so, by Remark \ref{rem:TMS}, $\cA\models\TMS{S}{\cR}$.
\end{proof}

We can now state and prove an updated version of Theorem \ref{prethm}.

\begin{theorem}\label{thm:models}
Let $\cR$ be a distance magma and fix $S\seq R$, with $0\in S$. There is an $\LDS$-structure $\cS^*=(S^*,\cps,\sleq,0)$ satisfying the following properties.
\begin{enumerate}[$(a)$]
\item $\cS^*$ is a distance magma.
\item $(S^*,\sleq)$ is an extension of $(S,\leq)$, and $S$ is dense in $S^*$ (with respect to the order topology).
\item For all $r,s\in S$, if $r\p s\in S$ then $r\cps s=r\p s$.
\item Suppose $\cA$ is an $\cR$-metric space, with $\Delta(\cA)\subsetsim\Delta(S,\cR)$. If $M\models \Th_{\cL_S}(\cA)$, then $M$ is $\cS^*$-colorable and $(M,d_M)$ is an $\cS^*$-metric space.
\end{enumerate}
\end{theorem}
\begin{proof}
Parts $(a)$, $(b)$, and $(c)$ follow from Theorem \ref{thm:S*DM}, Proposition \ref{dense}$(a)$, and Lemma \ref{lem:explSum}$(d)$, respectively. For part $(d)$, we have $\TMS{S}{\cR}\seq\Th_{\cL}(\cA)$ by Proposition \ref{convAx}, and so the statements follow from Proposition \ref{axModels}.
\end{proof}

Much of the previous work relied on approximating $\cS^*$-triangles with triangles in $\Delta(S,\cR)$. We now extend this notion of approximation to larger $\cS^*$-colored spaces.

\begin{definition}
Let $\cR$ be a distance magma and fix $S\seq R$, with $0\in S$. An $\cS^*$-colored space $(A,d_A)$ is \textbf{approximately $(S,\cR)$-metric} if, for all finite $A_0\seq A$ and all $S$-approximations $\Phi$ of $\Dist(A_0,d_A)$, there is an $\cR$-metric $d_\Phi$ on $A_0$ such that $d_\Phi(a,b)\in\Phi(d_A(a,b))\cap S$ for all $a,b\in A_0$.
\end{definition} 

\begin{proposition}\label{approximable}
Let $\cR$ be a distance magma and fix $S\seq R$, with $0\in S$. Suppose $\cA=(A,d_A)$ is an $\cS^*$-colored space. If $\cA$ is approximately $(S,\cR)$-metric then $\cA$ is an $\cS^*$-metric space.
\end{proposition}
\begin{proof}
Suppose $\cA$ is approximately $(S,\cR)$-metric. By compactness and Proposition \ref{convAx}, $\cA$ is an $\cL_S$-substructure of some model of $\TMS{S}{\cR}$, which is an $\cS^*$-metric space by Proposition \ref{axModels}$(b)$. So $\cA$ is an $\cS^*$-metric space.
\end{proof}

Regarding the converse of this fact, we have shown that $\cS^*$-metric spaces, with at most three points, are approximately $(S,\cR)$-metric (combine Propositions \ref{prop:approx}$(a)$ and \ref{prop:expl}). For larger $\cS^*$-metric spaces, this can fail.

\begin{example}
Let $\cR=(\R^{\geq0},+,\leq,0)$ and $S=[0,2)\cup[3,\infty)$. By Lemma \ref{lem:explSum}, we have $1+^*_S 3=4$ and $1+^*_S 1=3$. Define the $\cS^*$-metric space $\cA$, where $A=\{w,x,y,z\}$, $d_A(w,x)=d_A(x,z)=d_A(w,y)=1$, $d_A(x,y)=d_A(w,z)=3$, and $d_A(y,z)=4$.  
Then the $S$-approximation of $\Dist(\cA)$, given by $\Phi(1)=(0,1]$, $\Phi(3)=(0,3]$, and $\Phi(4)=(3,4]$, witnesses that $\cA$ is not approximately $(S,\cR)$-metric.
\end{example} 

In the next section, we will isolate a natural assumption on $S$ under which the converse of Proposition \ref{approximable} holds.

\section{Completeness Properties for Distance Sets}\label{sec:closure}

Until this point, we have made no assumptions on the set of distances $S$ in a distance magma $\cR$, other than $0\in S$. In this section, we formulate certain properties of distance sets which allow for suitable analogs of ``addition of distances" and ``absolute value of the difference between distances".

\subsection{Sum-completeness}

\begin{definition}
Let $\cR$ be a distance magma. A subset $S\seq R$ is \textbf{sum-complete in $\cR$} if $0\in S$ and, for all $r,s\in S$, the set $\{x\in S:x\leq r\p s\}$ contains a maximal element. In this case, we define $r\ps s=\max\{x\in S:x\leq r\p s\}$ and we let $\cS$ denote the $\LDS$-structure $(S,\ps,\leq,0)$.
\end{definition}

We omit the clause ``in $\cR$" when the ambient distance magma is clear from context. Note that the consideration of sum-complete subsets of $\cR$ generalizes Example \ref{introEx}$(3)$. The canonical examples of sum-complete subsets of $R$ are sets which contain $0$ and are closed under $\p$. For example, $R$ itself is always sum-complete in $\cR$. Any finite subset of $R$ containing $0$ is sum-complete. If the ordering on $\cR$ is complete then any closed subset of $R$ containing $0$ is sum-complete. The main property of sum-complete sets is that they admit a distance magma structure.

\begin{proposition}\label{coincides}
Given a distance magma $\cR$, $S\seq R$ is sum-complete if and only if $0\in S$ and, for all $r,s\in S$, $r\cps s\in S$ and $r\cps s\leq r\p s$. In this case, $r\cps s=r\ps s$ for all $r,s\in S$, and $\cS$ is a distance magma.
\end{proposition}
\begin{proof}
This follows easily from Lemma \ref{lem:explSum}$(b)$.
\end{proof}

We also note the following corollary of Proposition \ref{associative*}, which will be helpful when we eventually focus on distance \emph{monoids}.

\begin{corollary}\label{assoc*cor}
Suppose $\cR$ is a distance magma and $S\seq R$ is sum-complete. Then $\cS^*$ is a distance monoid if and only if $\cS$ is a distance monoid.
\end{corollary}

\begin{remark}\label{simplify}
Suppose $\cR$ is a distance magma and $S\seq R$ is sum-complete.
\begin{enumerate}
\item Note that we may construct $\cS^*$ while viewing $S$ as a subset of the distance magma $\cS$. Using Notation \ref{S*expl}, Lemma \ref{lem:explSum}, and Proposition \ref{coincides}, it is straightforward to verify that the resulting distance magma $\cS^*$ does not depend on this choice of context. Note also that an $\cS^*$-metric space is approximately $(S,\cR)$-metric if and only if it is approximately $(S,\cS)$-metric. 
\item By Theorem \ref{thm:models}$(b)$ and Proposition \ref{coincides}, we may consider $\cS$ as an $\LDS$-substructure of $\cS^*$. However, $\cS$ is usually not an elementary substructure. In fact, one may show that $\cS\preceq\cS^*$ if and only if  $\cS$ is well-ordered with a maximal element, in which case $\cS=\cS^*$.
\end{enumerate}
\end{remark}

The goal of this subsection is the converse of Proposition \ref{approximable} for sum-complete sets. We first define certain well-behaved $S$-approximations.  


\begin{definition}\label{def:metSol}
Let $\cR$ be a distance magma. Assume $S\seq R$ is sum-complete and fix $X\seq S^*$.
\begin{enumerate}
\item $X$ is \textbf{$S$-bounded} if for all $\alpha\in X$ there is $s\in S$ with $\alpha\sleq s$ (i.e. if $\omega_S\in S$ or $\omega_S\not\in X$).
\item Suppose $X$ is $S$-bounded and $\Phi$ is an $S$-approximation of $X$. Then $\Phi$ is \textbf{metric} if $\Phi^+(X)\seq	 S$ and, for all $\alpha,\beta,\gamma\in X$, if $\alpha\sleq\beta\cps\gamma$ then $\Phi^+(\alpha)\leq\Phi^+(\beta)\ps \Phi^+(\gamma)$.

\item If $\Phi$ and $\Psi$ are $S$-approximations of $X$ then $\Phi$ \textbf{refines} $\Psi$ if $\Phi(\alpha)\seq\Psi(\alpha)$ for all $\alpha\in X$. 
\end{enumerate}
\end{definition}

We now give the main results concerning sum-complete sets. 

\begin{lemma}\label{metSol}
Let $\cR$ be a distance magma and fix a sum-complete subset $S\seq R$. Suppose $X\seq S^*$ is finite and $S$-bounded. For any $S$-approximation $\Psi$ of $X$ there is a metric $S$-approximation $\Phi$ of $X$, which refines $\Psi$. 
\end{lemma}
\begin{proof}
For convenience, assume $0\in X$. Let $X=\{\alpha_0,\alpha_1,\ldots,\alpha_n\}$, with $0=\alpha_0\sle\alpha_1\sle\ldots\sle\alpha_n$. Fix an $S$-approximation $\Psi$ of $X$. Since $X$ is $S$-bounded, we may assume $\Psi^+(X)\seq S$. By density of $S$, we may also assume $\Psi^+(\alpha_k)\sle\alpha_{k+1}$ for all $1\leq k<n$. Given $1\leq k\leq n$, define 
\[J_k=\{(i,j):1\leq i, j<k,~\alpha_k\sleq\alpha_i\cps\alpha_j\}.
\]
We inductively define $s_0,s_1,\ldots,s_n\in S$ such that 
\begin{enumerate}[(1)]
\item $\alpha_k\sleq s_k\leq\Psi^+(\alpha_k)$ for all $1\leq k\leq n$;
\item for all $1\leq k\leq n$, if $(i,j)\in J_k$ then $s_k\leq s_i\ps s_j$.
\end{enumerate} 
Let $s_0=0$. Fix $1\leq k\leq n$ and suppose we have defined $s_i$ for all $1\leq i<k$. Define
$$
s_k=\min(\{\Psi^+(\alpha_k)\}\cup\{s_i\ps s_j:(i,j)\in J_k\}).
$$
Then $(2)$ is satisfied. For $(1)$, we have $s_k\leq\Psi^+(\alpha_k)$, so it remains to show $\alpha_k\sleq s_k$. Given $(i,j)\in J_k$, we have, by induction, $\alpha_k\sleq\alpha_i\cps\alpha_j\sleq s_i\cps s_j=s_i\ps s_j$. 

Define $\Phi:X\func \Int^*(S)$ such that $\Phi(0)=\{0\}$ and, for $k>0$, 
$\Phi(\alpha_k)=(\Psi^-(\alpha_k),s_k]$. Then, by $(1)$, $\Phi$ is an $S$-approximation of $X$, which refines $\Psi$. So it remains to show $\Phi$ is metric. Fix $\alpha_i,\alpha_j,\alpha_k\in X$ such that $\alpha_k\sleq\alpha_i\cps\alpha_j$. We want to show $s_k\leq s_i\ps s_j$. Since $(s_i)_{i=0}^k$ is increasing by construction, we may assume $i,j<k$. Then $(i,j)\in J_k$, and so $s_k\leq s_i\ps s_j$ by $(3)$.
\end{proof}

\begin{theorem}\label{metCor}
Let $\cR$ be a distance magma and fix a sum-complete subset $S\seq R$. Suppose $\cA=(A,d_A)$ is an $\cS^*$-colored space. Then $\cA$ is an $\cS^*$-metric space if and only if $\cA$ is approximately $(S,\cR)$-metric. 
\end{theorem}
\begin{proof}
We have the reverse direction by Proposition \ref{approximable}. For the forward direction, assume $\cA$ is an $\cS^*$-metric space. Fix a finite subset $A_0\seq A$ and an $S$-approximation $\Phi$ of $\Dist(A_0,d_A)$. We want to find an $\cR$-metric $d_\Phi:A_0\times A_0\func S$ such that, for all $x,y\in A_0$, $d_\Phi(x,y)\in\Phi(d_A(x,y))$. 

Suppose first that $\Dist(A_0,d_A)$ is not $S$-bounded. Then we may fix $t\in S$, with $\Phi^-(\omega_S)<t$ and $\alpha\sleq t$ for all $\alpha\in\Dist(A_0,d_A)\backslash\{\omega_S\}$, and define $d'_A:A_0\times A_0\func S^*$ such that
$$
d'_A(x,y)=\begin{cases}
d_A(x,y) & \text{if $d_A(x,y)\sle\omega_S$}\\
t & \text{otherwise.}
\end{cases}
$$
Then $d'_A$ is an $\cS^*$-metric on $A_0$, $\Dist(A_0,d'_A)$ is $S$-bounded, and $\Phi$ is still an $S$-approximation of $\Dist(A_0,d'_A)$. Therefore, without loss of generality, we assume $\Dist(A_0,d_A)$ is $S$-bounded.

By Lemma \ref{metSol}, there is a metric $S$-approximation $\Phi_0$ of $\Dist(A_0,d_A)$, which refines $\Phi$. Define $d_\Phi:A_0\times A_0\func S$ such that $d_\Phi(x,y)=\Phi_0^+(d_A(x,y))$.
\end{proof}

\subsection{Difference-completness}

We now define a property of distance magmas $\cR$, under which we can define a generalized notion of ``absolute value of the difference between two distances". 

\begin{definition}
A distance magma $\cR$ is \textbf{difference-complete} if, for all $r,s\in R$, the set $\{x\in R:r\leq s\p x\text{ and }s\leq r\p x\}$ contains a minimal element. In this case we set $|r\m s|=\min\{x\in R:r\leq s\p x\text{ and }s\leq r\p x\}$.
\end{definition}

The next proposition shows that this generalized difference operation behaves like the usual absolute value operation in many ways.





\begin{proposition}\label{compProps}
Suppose $\cR$ is a difference-complete distance magma.
\begin{enumerate}[$(a)$]
\item For all $r,s\in R$, if $s\leq r$ then $|r\m s|=\inf\{x\in R:r\leq s\p x\}$.
\item For all $r,s,t\in R$, $|r\m s|\leq t$ if and only if $r\leq s\p t$ and $s\leq r\p t$.
\item For all $r,s\in R$, $|r\m s|\leq\max\{r,s\}\leq r\p s$.
\item For all $r,s\in R$, $|r\m s|=|s\m r|$, and $|r\m s|=0$ if and only if $r=s$.
\item Define $d:R\times R\func R$ such that $d(r,s)=|r\m s|$. Then $d$ is an $\cR$-metric on $R$ if and only if $\p$ is associative.
\end{enumerate}
\end{proposition}
\begin{proof}
Parts $(a)$ through $(d)$ follow trivially from the definitions. We leave part $(e)$ as an exercise.
\end{proof}

\begin{remark}\label{rem:USC}
A consequence of part $(b)$ is that if $\cR$ is difference-complete then, for any $r\in R$, the function $f:x\mapsto x\p r$ is \emph{upper semi-continuous}: for all $x_0\in R$ and $s>f(x_0)$, there is $t>x_0$ such that $f(x)<s$ for all $x<t$ (simply take $t=|r\m s|$). 
\end{remark}

The main result of this section  is that the distance magma $\cS^*$ is \emph{always} difference-complete.

\begin{theorem}\label{S*complete}
Suppose $\cR$ is a distance magma and $S\seq R$, with $0\in S$. Then $\cS^*$ is a difference-complete distance magma.
\end{theorem}
\begin{proof}
Fix $\alpha,\beta\in S^*$ and let $\gamma=\inf\{x\in S^*:\alpha\sleq\beta\cps x\text{ and }\beta\sleq\alpha\cps x\}$. We want to show $\alpha\sleq\beta\cps\gamma$ and $\beta\sleq\alpha\cps\gamma$. If $\gamma\in S$ then we have the result by Proposition \ref{dense}$(b)$. So assume $\gamma\not\in S$. Without loss of generality, we may assume $\beta\sleq\alpha$. So we just need to show $\alpha\sleq\beta\cps\gamma$.  If $\beta\cps\gamma\sle\alpha$ then by Lemma \ref{lem:explSum}$(c)$, there are $s,t\in S$ such that $\beta\sleq s$, $\gamma\sleq t$, and $s\cps t\sle\alpha$. Then $\beta\cps t\sle \alpha$ and so $t\sleq\gamma$ by definition of $\gamma$. But then $\gamma=t\in S$, which contradicts our assumptions.
\end{proof}

For clarity, we repeat the generalized difference operation on $\cS^*$.

\begin{definition}\label{def:minus}
Fix a distance magma $\cR$ and $S\seq R$, with $0\in S$. Given $\alpha,\beta\in S^*$, define
$$
|\alpha\cms\beta|:=\inf\{x\in S^*:\alpha\sleq\beta\cps x\text{ and }\beta\sleq\alpha\cps x\}. 
$$
\end{definition}

Recall that $\alpha\cps\beta$ is the largest possible distance in a logical $S^*$-triangle containing distances $\alpha$ and $\beta$. Combining Theorem \ref{S*complete} with Proposition \ref{prop:expl}, we see that $|\alpha\cms\beta|$ is the smallest possible distance.

\begin{corollary}\label{Delta*}
Suppose $\cR$ is a distance magma and $S\seq R$, with $0\in S$. Given $\alpha,\beta\in S^*$, we have $\Sigma(\alpha,\beta)=\{\gamma\in S^*:|\alpha\cms\beta|\sleq\gamma\sleq\alpha\cps\beta\}$, and so $|\alpha\cms\beta|=\inf\Sigma(\alpha,\beta)$.
\end{corollary}

\section{Associativity, Amalgamation, and the Four-Values Condition}\label{sec:4VC}

We have now laid the foundation for the model theoretic study of generalized metric spaces, and the next task is to find concrete spaces to study. A natural choice is to consider ``generic objects", in the sense of homogeneous structures and \Fraisse\ limits. In particular, our motivating example is the \textit{rational Urysohn space}, i.e. the unique countable, universal, and ultrahomogeneous metric space with rational distances. In \cite{DLPS}, generalizations of this space are obtained by replacing $\Q^{\geq0}$ with other countable subsets $S\seq\R^{\geq0}$. The sets $S$ for which an analogous metric space exists are characterized in \cite{DLPS} by a property called the \textit{four-values condition}. 

We first generalize the four-values condition to arbitrary distance magmas. Our treatment closely follows \cite{DLPS}. In particular, Proposition \ref{DLPS1.6}, which is the main result of this section, is a direct generalization of \cite[Section 1.3]{DLPS}. Throughout the section, we fix a distance magma $\cR=(R,\p,\leq,0)$.

\begin{definition}
A subset $S\seq R$ satisfies the \textbf{four-values condition in $\cR$} if for all $u_1,u_2,v_1,v_2\in S$, if there is some $s\in S$ such that $(s,u_1,u_2)$ and $(s,v_1,v_2)$ are $\cR$-triangles, then there is some $t\in S$ such that $(t,u_1,v_1)$ and $(t,u_2,v_2)$ are $\cR$-triangles.
\end{definition}

\begin{figure}[htbp]
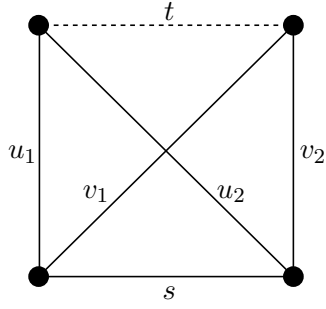

\begin{center}
\begin{lpic}[t(0mm),b(0mm),draft,clean]{crown(.75)}
\lbl[t]{-1,25;$u_1$}
\lbl[t]{36,18;$u_2$}
\lbl[t]{12,18;$v_1$}
\lbl[t]{50.5,25;$v_2$}
\lbl[t]{25,0;$s$}
\lbl[t]{25,50.5;$t$}
\end{lpic}
\end{center}
\caption{the four-values condition}
\label{4VCfig}
\end{figure}

The four-values condition describes the amalgamation of two $3$-point metric spaces over a common $2$-point subspace (Figure \ref{4VCfig}). In Proposition \ref{DLPS1.6}, we show that this instance of amalgamation is enough to obtain amalgamation for any two finite $\cR$-metric spaces with distances in $S$.

\begin{definition}
Given $S\seq R$, with $0\in S$, let $\KRS{S}{\cR}$ denote the class of finite $\cR$-metric spaces with distances in $S$. Let $\cK_\cR$ denote $\KRS{R}{\cR}$.
\end{definition}

Given a distance magma $\cR$ and a subset $S\seq R$, with $0\in S$, we use our original interpretation of $\cR$-metric spaces as $\cL_S$-structures to view $\KRS{S}{\cR}$ as a class of relational structures amenable to classical \Fraisse\ theory (see \cite[Chapter 7]{Hobook}). In particular, it is straightforward to see that the class $\KRS{S}{\cR}$ always satisfies the hereditary property and the joint embedding property. Therefore, our focus is on the amalgamation property. 

The next result uses the four-values condition to characterize the amalgamation property for $\KRS{S}{\cR}$. This result is a direct generalization of \cite[Proposition 1.6]{DLPS}. The proof is the same as what can be found in \cite{DLPS}, modulo adjustments made to account for the possibility that $\cR$ is not difference-complete. We include the steps requiring these adjustments, and refer the reader to \cite{DLPS} for the remaining details.

\begin{proposition}\label{DLPS1.6}
Fix $S\seq R$, with $0\in S$. Then $\KRS{S}{\cR}$ has the (disjoint) amalgamation property if and only if $S$ satisfies the four-values condition in $\cR$.
\end{proposition}
\begin{proof}
If $\KRS{S}{\cR}$ has the amalgamation property then the proof that $S$ satisfies the four-values condition in $\cR$ follows exactly as in \cite[Proposition 1.6]{DLPS}. The essential idea is to consider Figure \ref{4VCfig}. For the converse, we assume $S$ satisfies the four-values condition in $\cR$, and prove that $\KRS{S}{\cR}$ has the disjoint amalgamation property. Fix $(X_1,d_1)$ and $(X_2,d_2)$ in $\KRS{S}{\cR}$ such that $d_1|_{X_1\cap X_2}=d_2|_{X_1\cap X_2}$ and $X_1\cap X_2\neq\emptyset$. We may assume $X_1\not\seq X_2$ and $X_2\not\seq X_1$. Let $m=|(X_1\backslash X_2)\cup (X_2\backslash X_1)|$ and set $X=X_1\cup X_2$. Then $m\geq 2$ by our assumptions, and we proceed by induction on $m$. 

Suppose $m=2$. Let $X_1\backslash X_2=\{x_1\}$ and $X_2\backslash X_1=\{x_2\}$. Given $t\in S$, let $d_t:X\times X\func S$ be such that $d_t|_{X_1}=d_1$, $d_t|_{X_2}=d_2$, and $d_t(x_1,x_2)=t$. Then $d_t$ is an $\cR$-metric on $X$ if and only if 
\begin{equation*}
\text{$t>0$ and $(t,d_1(x_1,x),d_2(x_2,x))$ is an $\cR$-triangle for all $x\in X_1\cap X_2$.}\tag{$\dagger$}
\end{equation*}
Therefore, it suffices to find $t\in S$ satisfying $(\dagger)$.

Fix $y\in X_1\cap X_2$ such that
$$
d_1(x_1,y) \p d_2(x_2,y) = \min_{x\in X_1\cap X_2} (d_1(x_1,x)\p d_2(x_2,x)).
$$
Given $r,s\in R$, let $D(r,s)=\{x\in R:r\leq s\p x\text{ and }s\leq r\p x\}$. Note that $D(r,s)$ is closed upward in $R$ for any $r,s\in R$. Therefore we may fix $y'\in X_1\cap X_2$ such that 
$$
D(d_1(x_1,y'),d_2(x_2,y'))=\bigcap_{x\in X_1\cap X_2}D(d_1(x_1,x),d_2(x_2,x)).
$$
Then $(d_1(y,y'),d_1(x_1,y),d_1(x_1,y'))$ and $(d_2(y,y'),d_2(x_2,y),d_2(x_2,y'))$ are $\cR$-triangles in $S$. Since $d_1(y,y')=d_2(y,y')$ and $S$ satisfies the four-values condition in $\cR$, there is some $t\in S$ such that $(t,d_1(x_1,y),d_2(x_2,y))$ and $(t,d_1(x_1,y'),d_2(x_2,y'))$ are $\cR$-triangles. If $t=0$ then, after replacing $t$ with $\min\{d_1(x_1,y),d_1(x_1,y')\}$, we may assume $t>0$. Since $(t,d_1(x_1,y),d_2(x_2,y))$ is an $\cR$-triangle, we have
$$
t \leq d_1(x_1,y)\p d_2(x_2,y)= \min_{x\in X_1\cap X_2}(d_1(x_1,x) \p d_2(x_2,x)).
$$
Therefore, to show that $t$ satisfies $(\dagger)$, it remains to show that for all $x\in X_1\cap X_2$, we have $d_1(x_1,x)\leq d_2(x_2,x)\p t$ and $d_2(x_2,x)\leq d_1(x_1,x)\p t$. Since $(t,d_1(x_1,y'),d_2(x_2,y'))$ is an $\cR$-triangle, we have $t\in D(d_1(x_1,y'),d_2(x_2,y'))$. Therefore, by choice of $y'$, we have $t\in D(d_1(x_1,x),d_2(x_2,x))$ for all $x\in X_1\cap X_2$, which yields the desired inequalities. This completes the base case $m=2$. The induction step proceeds exactly as in \cite[Proposition 1.6]{DLPS}.
\end{proof}
 
Using the previous characterization, we proceed as follows. Fix a distance magma $\cR$ and a subset $S\seq R$, with $0\in S$. In order to apply classical \Fraisse\ theory, we assume $S$ is countable, which means $\KRS{S}{\cR}$ is a countable (up to isomorphism) class of $\cL_S$-structures. If we also assume $S$ satisfies the four-values condition in $\cR$ then, altogether, $\KRS{S}{\cR}$ is a \Fraisse\ class and so we may define the \Fraisse\ limit (see \cite[Theorem 7.1.2]{Hobook}).

\begin{definition}
Given a distance magma $\cR$ and a countable subset $S\seq R$, such that $0\in S$ and $S$ satisfies the four-values condition in $\cR$, let $\cU^S_\cR$ denote the \Fraisse\ limit of $\KRS{S}{\cR}$. Let $\cU_\cR$ denote $\cU^R_\cR$.
\end{definition}

We now obtain a countable $\cL_S$-structure $\cU^S_\cR$, and it is clear that $\cU^S_\cR\models \TMS{S}{\cR}$. By Proposition \ref{axModels}, $\cU^S_\cR$ is an $\cS^*$-metric space. However, since the age of $\cU^S_\cR$ is precisely $\KRS{S}{\cR}$, it follows that $\cU^S_\cR$ is an $\cR$-metric space with $\Dist(\cU^S_\cR)=S$. Altogether, given a distance magma $\cR$ and a countable subset $S\seq R$, such that $0\in S$ and $S$ satisfies the four-values condition in $\cR$, we call $\cU^S_\cR$ the \textbf{$\cR$-Urysohn space over $S$}.

We summarize with the following combinatorial description of $\URS{S}{\cR}$.

\begin{theorem}
Fix a distance magma $\cR$ and a countable set $S\seq R$, with $0\in S$. 
\begin{enumerate}[$(a)$]
\item If $S$ satisfies the four-values condition in $\cR$ then $\cU^S_\cR$ is the unique (up to isometry) $\cR$-metric space satisfying the following properties:
\begin{enumerate}[$(i)$]
\item $\cU^S_\cR$ is countable and $\Dist(\cU^S_\cR)=S$;
\item (\textit{ultrahomogeneity}) any isometry between two finite subspaces of $\cU^S_\cR$ extends to a total isometry of $\cU^S_\cR$;
\item (\textit{universality}) any element of $\KRS{S}{\cR}$ is isometric to a subspace of $\cU^S_\cR$.
\end{enumerate}
\item Suppose there is a countable, ultrahomogeneous $\cR$-metric space $\cA$, such that $\Dist(\cA)=S$ and $\cA$ is universal for $\KRS{S}{\cR}$. Then $S$ satisfies the four-values condition in $\cR$ and so $\cA$ is isometric to $\cU^S_\cR$.
\end{enumerate}
\end{theorem}

\begin{remark}$~$
\begin{enumerate}
\item Consider the distance monoid $\cQ=(\Q^{\geq0},+,\leq,0)$. Then $\cU_\cQ$ is precisely the classical \textit{rational Urysohn space}, which is an important example in model theory, descriptive set theory, Ramsey theory, and topological dynamics of isometry groups. The completion of the rational Urysohn space is called the \textit{Urysohn space}, and is the universal separable metric space. Both spaces were first constructed by Urysohn in 1925 (see \cite{Ury}, \cite{Ury2}). Further details and results can be found in \cite{MeUS}.
\item In Proposition \ref{DLPS1.6}, there is no restriction on the cardinality of $S$. However, in order to apply classical \Fraisse\ theory and construct a countable space $\URS{S}{\cR}$, we must assume $S$ is countable. In \cite{Sa13}, Sauer considers arbitrary subsets $S\seq\R^{\geq0}$ and, combining the four-values condition with certain topological properties, characterizes the existence of a universal separable complete metric space with distances in $S$ (e.g. if $S=\R^{\geq0}$ then this produces the Urysohn space).
\end{enumerate}
\end{remark}

Note that if $S\seq R$ is countable and sum-complete then $\cK_\cS=\KRS{S}{\cR}$ and $\cU_\cS=\URS{S}{\cR}$. In this case, we have the following nice characterization of when $\cU_\cS$ exists. This result was first shown for (topologically) closed subsets of $(\R^{\geq0},+,\leq,0)$ by Sauer in \cite[Theorem 5]{Sa13b}, and the following is, once again, a direct generalization. 

\begin{proposition}\label{4VCassoc}
Suppose $S\seq R$ is sum-complete. Then $S$ satisfies the four-values condition in $\cR$ if and only if $\ps$ is associative on $S$.
\end{proposition}
\begin{proof}
Suppose $S$ satisfies the four-values condition in $\cR$, and fix $r,s,t\in S$. Since $\ps$ is commutative, it suffices to show $(r\ps s)\ps t\leq r\ps (s\ps t)$. Let $u=(r\ps s)\ps t$. Then $(r\p_S s,r,s)$ and $(r\p_S s,u,t)$ are both $\cR$-triangles. By the four-values condition, there is $v\in S$ such that $(v,r,u)$ and $(v,s,t)$ are $\cR$-triangles. Therefore $u\leq r\ps v\leq r\ps (s\ps t)$, as desired.

Conversely, assume $\ps$ is associative on $S$. Fix $u_1,u_2,v_1,v_2,s\in S$ such that $(s,u_1,u_2)$ and $(s,v_1,v_2)$ are $\cR$-triangles. Without loss of generality, assume $u_1\p v_1\leq u_2\p v_2$. Let $t=u_1\ps v_1$. Then $(t,u_1,v_1)$ is clearly an $\cR$-triangle, so it suffices to show that $(t,u_2,v_2)$ is an $\cR$-triangle. We have $t\leq u_1\p v_1\leq u_2\p v_2$ by assumption, so it remains to show $v_2\leq u_2\p t$ and $u_2\leq v_2\p t$. Note that $s\leq u_2\ps u_1$ and $v_2\leq s\ps v_1$ since $(s,u_1,u_2)$ and $(s,v_1,v_2)$ are $\cR$-triangles. Therefore $v_2\leq s\ps v_1\leq (u_2\ps u_1)\ps v_1=u_2\ps (u_1\ps v_1)\leq u_2\p t$.
The argument for $u_2\leq v_2\p t$ is similar.
\end{proof}

In particular, this says that if $\cR$ is a distance magma, then $R$ satisfies the four-values condition in $\cR$ if and only if $\cR$ is a distance monoid. By Corollary \ref{assoc*cor} and Proposition \ref{4VCassoc}, we also obtain the following corollary.

\begin{corollary}
If $S\seq R$ is sum-complete and satisfies the four-values condition in $\cR$, then $\cS^*$ is a distance monoid.
\end{corollary}

\begin{example}
We show that, in the previous corollary, the sum-completeness of $S$ is necessary. Let $\cR=(\R^{\geq0},+,\leq,0)$ and $S=[0,2)\cup(4,\infty)$. Note that $S$ is not sum-complete, witnessed by $\{x\in S:x\leq 1\p 1\}$. The reader may verify that $S$ satisfies the four-values condition in $\cR$.
On the other hand, $+^*_S$ is not associative on $S^*$. Indeed, if $X$ is the gap cut $(4,\infty)$ then, by Lemma \ref{lem:explSum}, $(1+^*_S 1)+^*_S g_X=g_X+^*_S g_X=8^+$ and $1+^*_S(1+^*_S g_X)=1+^*_S 5^+=6^+$.
\end{example}

Finally, it is worth noting that if $\cR$ is a countable distance \emph{monoid} then there is a more direct way to demonstrate that $\cK_\cR$ is a \Fraisse\ class. In particular, to prove $\cK_\cR$ has the amalgamation property, one may use the natural generalization of the notion of ``free amalgamation of metric spaces." 

\begin{definition} Let $\cR$ be a distance magma.
\begin{enumerate}
\item Suppose $\cA=(A,d_A)$ and $\cB=(B,d_B)$ are finite $\cR$-metric spaces such that $A\cap B\neq\emptyset$ and $d_A|_{A\cap B}=d_B|_{A\cap B}$. Define the $\cR$-colored space $\cA\otimes\cB=(C,d_C)$ where $C=A\cup B$ and
$$
d_C(x,y)=\begin{cases}
d_A(x,y) & \text{if $x,y\in A$}\\
d_B(x,y) & \text{if $x,y\in B$}\\
\displaystyle\min_{z\in A\cap B}(d_A(x,z)\p d_B(z,y)) & \text{if $x\in A\backslash B$ and $y\in B\backslash A$.}
\end{cases}
$$
\item $\cR$ \textbf{admits free amalgamation of metric spaces} if $\cA\otimes\cB$ is an $\cR$-metric space for all finite $\cR$-metric spaces $\cA$ and $\cB$. 
\end{enumerate}
\end{definition}

\begin{exercise}\label{freeAmal}
Let $\cR$ be a distance magma. Then $\cR$ admits free amalgamation of metric spaces if and only if $\p$ is associative.
\end{exercise}

\section{Quantifier Elimination in Generalized Urysohn Spaces}\label{sec:QE}

In this section, we consider quantifier elimination in the theory of a generalized Urysohn space of the kind constructed in Section \ref{sec:4VC}. The setup is as follows. We have a distance magma $\cR=(R,\p,\leq,0)$ and a countable subset $S\seq R$, such that $0\in S$ and $S$ satisfies the four-values condition in $\cR$. We will also assume $S$ is sum-complete. The reason for this is that Lemma \ref{extAx}, which is a key tool in this section, crucially relies on the existence of an associative binary operation on the set of distances. Altogether, in light of Remark \ref{simplify}(1), we can encompass this general situation by simply fixing a \emph{countable distance monoid} $\cR=(R,\p,\leq,0)$ and choosing $S=R$. 

Let $\TUS{\cR}$ denote the complete $\cL_R$-theory of  the $\cR$-Urysohn space $\cU_\cR$ (which exists by Proposition \ref{4VCassoc}), and let $d$ denote $d_{\cU_\cR}$. Then $\cR^*=(R^*,\cp_R,\leq^*,0)$ is a difference-complete distance monoid, with generalized difference operation $|\alpha \cm_R\beta|$ (see Definition \ref{def:minus}). We continue to consider $\cR$ as an $\LDS$-substructure of $\cR^*$. Therefore, to ease notation, we will omit the extra decorations on the symbols in $\LDS$, and let $\cR^*=(R^*,\p,\leq,0)$. We also use $|\alpha \m\beta|$ to denote $|\alpha\cm_R\beta|$. It is worth observing that, while $\cR^*$ is difference-complete by Theorem \ref{S*complete}, $\cR$ itself may not be difference-complete and so $R$ may not be closed under $|r\m s|$.

The next claim follows by universality of $\cU_\cR$, Theorem \ref{metCor}, and compactness.

\begin{proposition}\label{saturated2}
Any $\cR^*$-metric space is isometric to a subspace of some model of $\TUS{\cR}$.
\end{proposition}

The goal of this section is Theorem \ref{prethm2}, a characterization of quantifier elimination for $\TUS{\cR}$. The proof will rely on \textit{extension axioms}, i.e. $\cL_R$-sentences approximating one-point extensions of finite $\cR^*$-metric spaces. We begin with several definitions in this direction.

\begin{definition}
Fix an $\cR^*$-metric space $\cA=(A,d_A)$. 
\begin{enumerate}
\item A function $f:A\func R^*$ is an \textbf{$\cR^*$-Kat\v{e}tov map on $\cA$} if, for all $x,y\in A$, $(d_A(x,y),f(x),f(y))$ is an $\cR^*$-triangle.
\item Let $E_{\cR^*}(\cA)$ be the set of $\cR^*$-Kat\v{e}tov maps on $\cA$.
\end{enumerate}
\end{definition}

\begin{remark} Note that the definition of Kat\v{e}tov map makes sense in the context of an arbitrary distance magma. These maps take their name from \cite{Kat}, in which Kat\v{e}tov uses them to construct the Urysohn space, as well as similar metric spaces in larger cardinalities. See \cite{MeUS} for more on Kat\v{e}tov maps in the classical setting, including an analysis of $E_\cR(\cA)$ as a topological space. 

Kat\v{e}tov maps have a natural model theoretic characterization as quantifier-free $1$-types. In particular, if $\cA$ is an $\cR^*$-metric space then, by Proposition \ref{saturated2}, we may fix $M\models \Th(\cU_\cR)$ such that $\cA$ is a subspace of $(M,d_M)$. Let $S^{\text{qf}}_1(A)$ be the space of quantifier-free $1$-types over the parameter set $A$. Given $f\in E_{\cR^*}(\cA)$, define $q_f(x)=\bigcup_{a\in A}p_{f(a)}(x,a)$. Conversely, given $q(x)\in S_1^{\text{qf}}(A)$, let $f_q:A\func R^*$ such that $p_{f_q(a)}(x,a)\seq q(x)$. Then one may verify that $f\mapsto q_f$ is a well-defined bijection from $E_{\cR^*}(\cA)$ to $S_1^{\text{qf}}(A)$, with inverse $q\mapsto f_q$.
\end{remark}

Going forward, we will only consider \textit{non-principal} Kat\v{e}tov maps, i.e. those not containing $0$ in their image. 

\begin{definition}
Fix an $\cR^*$-metric space $\cA=(A,d_A)$.
\begin{enumerate}
\item Let $E^+_{\cR^*}(\cA)=\{f\in E_{\cR^*}(\cA):f(a)>0\text{ for all $a\in A$}\}$.
\item Given $f\in E^+_{\cR^*}(\cA)$, define an $\cR^*$-metric space $\cA^f=(A^f,d_A)$ where $A^f=A\cup\{z_f\}$, with $z_f\not\in A$, and, for all $x\in A$, $d_A(x,z_f)=f(x)=d_A(z_f,x)$.
\end{enumerate}
\end{definition}

Next, we give a variation of the notion of $R$-approximation, which will simplify some steps of the arguments in this section.  Recall that $\Int^*(R)$ is the set of intervals in $\cR^*$ of the form $\{0\}$, $(r,s]$ for some $r,s\in R$, or $(r,\omega_R]$ where $r\in R$ and $\omega_R$ is the maximal element of $R^*$.

\begin{definition}
Fix a finite $\cR^*$-metric space $\cA=(A,d_A)$.
\begin{enumerate}
\item A symmetric function $\Phi:A\times A\func \Int^*(R)$ is an \textbf{$R$-approximation} of $\cA$ if $d_A(a,b)\in\Phi(a,b)$ for all $a,b\in A$. We write $\Phi(a,b)=(\Phi^-(a,b),\Phi^+(a,b)]$.
\item Given an $R$-approximation $\Phi$ of $\cA$ and $\alpha\in\Dist(\cA)$, define
\begin{align*}
\hat{\Phi}^-(\alpha) &= \max\{\Phi^-(a,b):d_A(a,b)=\alpha\}\mand\\
\hat{\Phi}^+(\alpha) &= \min\{\Phi^+(a,b):d_A(a,b)=\alpha\}.
\end{align*}
Let $\hat{\Phi}(\alpha)=(\hat{\Phi}^-(\alpha),\hat{\Phi}^+(\alpha)]$, and note that $\hat{\Phi}$ is an $R$-approximation of $\Dist(\cA)$ in the sense of Definition \ref{def:Sapprox} and Notation \ref{not:Sapprox}.
\item Given $f\in E^+_{\cR^*}(\cA)$, if $\Phi$ is an $R$-approximation of $\cA^f$ and $x\in A$, then we let $\Phi(x)=\Phi(x,z_f)$ and write $\Phi(x)=(\Phi^-(x),\Phi^+(x)]$.
\end{enumerate}
\end{definition}

\begin{definition}$~$
\begin{enumerate}
\item Given $I\in \Int^*(R)$, define the $\cL_R$-formula
$$
d(x,y)\in I:=\begin{cases}
r<d(x,y)\leq s & \text{if $I=(r,s]$ and $s\in R$}\\
d(x,y)>r & \text{if $I=(r,\omega_R]$ and $\omega_R\not\in R$}\\
x=y & \text{if $I=\{0\}$.}
\end{cases}
$$

\item Fix a finite $\cR^*$-metric space $\cA$ and $f\in E^+_{\cR^*}(\cA)$. Suppose $\Phi$ is an $R$-approximation of $\cA^f$. Let $A=\{a_1,\ldots,a_n\}$, and fix a tuple $\xbar=(x_1,\ldots,x_n)$ of variables. Define the $\cL_R$-formulas
\begin{align*}
C^\Phi_\cA(\xbar) &:= \bigwedge_{1\leq i,j\leq n} d(x_i,x_j)\in \Phi(a_i,a_j),\\
K^\Phi_\cA(\xbar,y) &:= \bigwedge_{1\leq i\leq n} d(x_i,y)\in\Phi(a_i),\mand\\
\epsilon^\Phi_\cA &:= \forall x_1\ldots x_n\bigg(C^\Phi_\cA(\xbar)\rightarrow\exists y K^\Phi_\cA(\xbar,y)\bigg).
\end{align*}
\end{enumerate}
\end{definition}

The sentences $\epsilon^\Phi_\cA$ should be viewed as extension axioms approximating Kat\v{e}tov maps. Informally, $\cU_\cR$ satisfies $\epsilon^\Phi_\cA$ if, for any $x_1,\ldots,x_n$ in $\cU_\cR$, if $\{x_1,\ldots,x_n\}$ is approximately isometric to $\cA$ then there is some $y$ in $\cU_\cR$ such that $\{x_1,\ldots,x_n,y\}$ is approximately isometric to $\cA^f$, where in both cases ``approximately isometric" is determined by $\Phi$. Of course, if $\Phi$ is a poor approximation of $\cA^f$ then there is no reason to expect $\cU_\cR\models\epsilon^\Phi_\cA$. This observation motivates our final definition. 

\begin{definition}$~$
\begin{enumerate}
\item An \textbf{extension scheme} is a triple $(\cA,f,\Psi)$, where $\cA$ is a finite $\cR^*$-metric space, $f\in E^+_{\cR^*}(\cA)$, and $\Psi$ is an $R$-approximation of $\cA^f$.
\item $\TUS{\cR}$ \textbf{admits extension axioms} if, for all extension schemes $(\cA,f,\Psi)$, there is an $R$-approximation $\Phi$ of $\cA^f$ such that $\Phi$ refines $\Psi$ and $\cU_\cR\models\epsilon^\Phi_\cA$. 
\end{enumerate}
\end{definition}

To avoid inconsequential complications when $\omega_R\not\in R$, we make the following reduction. Call an extension scheme $(\cA,f,\Psi)$ \textbf{standard} if $\Psi^+(A^f\times A^f)\seq R$.

\begin{proposition}\label{bddRed}
$\TUS{\cR}$ admits extension axioms if and only if, for all standard extension schemes $(\cA,f,\Psi)$, there is an $R$-approximation $\Phi$ of $\cA^f$ such that $\Phi$ refines $\Psi$ and $\cU_\cR\models\epsilon^\Phi_\cA$.
\end{proposition}
\begin{proof}
The forward direction is trivial. If $\omega_R\in R$ then every extension scheme is standard, and so the reverse direction is also trivial. Assume $\omega_R\not\in R$.

\noindent\textit{Claim}: If $\alpha,\beta\in R^*$ and $\max\{\alpha,\beta\}<\omega_R$, then $\alpha\p\beta<\omega_R$.

\noindent\textit{Proof}: By density of $R$, there are $r,s\in R$ such that $\alpha\leq r<\omega_R$ and $\beta\leq s<\omega_R$. So $\alpha\p\beta\leq r\p s<\omega_R$.\claim

Fix an extension scheme $(\cA,f,\Psi)$. Let $A_0=\{a\in A:f(a)< \omega_R\}$. Suppose $A_0=\emptyset$. We show $\cU_\cR\models\epsilon^\Psi_\cA$. Indeed, if $\cU_\cR\models C^\Psi_\cA(\bbar)$ and $s\in R$ is such that $d(b_i,b_j)\leq s$ for all $b_i,b_j\in\bbar$, then, by universality and homogeneity, there is some $c\in\cU_\cR$ such that $d(b_i,c)=s$ for all $b_i\in\bbar$. If, moreover, $\max\{\Psi^-(a):a\in A\}<s$, then $\cU_\cR\models K^\Psi_\cA(\bbar,c)$. So we may assume $A_0\neq\emptyset$.

Set $f_0=f|_{A_0}$, $\cA_0=(A_0,d_A)$, and $\Psi_0=\Psi|_{A_0^{f_0}\times A_0^{f_0}}$. From the claim, it follows that $d_A(a,b)<\omega_R$ for all $a,b\in A_0$, and so we may assume $(\cA_0,f_0,\Psi_0)$ is a standard extension scheme. By assumption, there is an $R$-approximation $\Phi_0$ of $\cA_0^{f_0}$ such that $\Phi_0$ refines $\Psi_0$ and $\cU_\cR\models\epsilon^{\Phi_0}_{\cA_0}$. We define an $R$-approximation $\Phi$ of $\cA^f$ such that, given $a,b\in A^f$,
$$
\Phi(a,b)=\begin{cases}
\Phi_0(a,b) & \text{if $a,b\in A_0\cup\{z_f\}$}\\
\hat{\Psi}(d_A(a,b)) & \text{otherwise.}
\end{cases}
$$
Then $\Phi$ refines $\Psi$, and we show $\cU_\cR\models\epsilon^\Phi_\cA$. Note, in particular, that if $a,b\in A^f$ and $d_A(a,b)=\omega_R$ then $\Phi(a,b)=\hat{\Psi}(\omega_R)$.

Let $A=\{a_1,\ldots,a_n\}$, with $A_0=\{a_1,\ldots,a_k\}$ for some $1\leq k\leq n$. Suppose $\bbar\in \cU_\cR$ is such that $\cU_\cR\models C^\Phi_\cA(\bbar)$. If $\bbar_0=(b_1,\ldots,b_k)$ then $\cU_\cR\models C^{\Phi_0}_{\cA_0}(\bbar_0)$ so there is some $c\in \cU_\cR$ such that $\cU_\cR\models K^{\Phi_0}_{\cA_0}(\bbar_0,c)$. By homogeneity of $\cU_\cR$ and Exercise \ref{freeAmal}, we may assume $c\not\in\bbar$ and $c\bbar$ is isometric to the free amalgamation $c\bbar_0\otimes \bbar$. To show $\cU_\cR\models K^\Phi_\cA(\bbar,c)$, it suffices to show $d(b_i,c)>\Phi^-(a_i)$ for all $k<i\leq n$. For this, given $k<i\leq n$, there is some $1\leq j\leq k$ such that $d(b_i,c)=d(b_i,b_j)\p d(b_j,c)$. Since $a_j\in A_0$ and $a_i\in A\backslash A_0$, we have $d_A(a_i,a_j)=\omega_R$ by the claim. Since $\cU_\cR\models C^\Phi_\cA(\bbar)$, we have $d(b_i,c)\geq d(b_i,b_j)>\Phi^-(a_i,a_j)=\hat{\Psi}^-(\omega_R)=\Phi^-(a_i)$.
\end{proof}

Next, we give sufficient conditions for when, in a standard extension scheme $(\cA,f,\Phi)$, $\Phi$ is a good enough approximation of $\cA^f$ to ensure $\cU_\cR\models\epsilon^\Phi_\cA$.

\begin{lemma}\label{extAx}
Suppose $(\cA,f,\Phi)$ is a standard extension scheme such that:
\begin{enumerate}[$(i)$]
\item for all $a,b\in A$, $\Phi^+(a,b)\leq \Phi^+(a)\p \Phi^+(b)$;
\item for all $a,b\in A$ and $t\in R$, if $\Phi^-(a,b)<t$ then $\Phi^-(a)<t\p \Phi^+(b)$.
\end{enumerate}
Then $\cU_\cR\models\epsilon^\Phi_\cA$.
\end{lemma}
\begin{proof}
Let $\{a_1,\ldots,a_n\}$ be an enumeration of $A$ such that $\Phi^+(a_1)\leq\ldots\leq\Phi^+(a_n)$. Suppose there are $b_1,\ldots,b_n\in \cU_\cR$ such that $\cU_\cR\models C^\Phi_\cA(\bbar)$. We will inductively construct $s_1,\ldots,s_n\in R$ such that:
\begin{enumerate}
\item $\Phi^-(a_k)<s_k\leq \Phi^+(a_k)$ for all $1\leq k\leq n$,
\item for all $1\leq k\leq n$, if $s_k<\Phi^+(a_k)$ then $s_k=s_i\p d(b_i,b_k)$ for some $i<k$,
\item for all $1\leq k\leq n$, if $i<k$ then $(d(b_i,b_k),s_i,s_k)$ is an $\cR$-triangle.
\end{enumerate}

Given this construction, let $g:\bbar\func R$ such that $g(b_i)=s_i$. Then $g\in E^+_{\cR^*}(\bbar,d)$ by $(3)$, with $\Dist(\bbar^g,d)\seq R$. By universality and homogeneity of $\cU_\cR$, there is some $c\in\cU_\cR$ such that $d(b_i,c)=s_i$ for all $1\leq i\leq n$. By $(1)$, $\cU_\cR\models K^\Phi_\cA(\bbar,c)$. Therefore, the above construction finishes the proof.

Let $s_1=\Phi^+(a_1)$. Fix $1<k\leq n$ and suppose we have $s_i$ for $i<k$. Define
$$
s_k=\min(\{\Phi^+(a_k)\}\cup\{s_i\p d(b_i,b_k):i<k\}).
$$
Note that $(2)$ is satisfied. We need to verify $(1)$ and $(3)$.

\noindent\textit{Case 1}: $s_k=\Phi^+(a_k)$. 

Then $(1)$ is satisfied. For $(3)$, note that for any $i<k$, we have
$$
s_k=\Phi^+(a_k)\leq s_i\p d(b_i,b_k)\mand s_i\leq\Phi^+(a_i)\leq\Phi^+(a_k)\leq s_k\p d(b_i,b_k).
$$
So we have left to fix $i<k$ and show $d(b_i,b_k)\leq s_i\p s_k$. Toward this end, we construct a sequence $i=i_0>i_1>\ldots> i_t$, for some $t\geq 0$, such that
\begin{enumerate}[\hspace{10pt}$\bullet$]
\item $s_{i_t}=\Phi^+(a_{i_t})$ and
\item for all $0\leq l<t$, $s_{i_l}=s_{i_{l+1}}\p d(b_{i_l},b_{i_{l+1}})$.
\end{enumerate}
Note that such a sequence exists by $(2)$, and since $s_1=\Phi^+(a_1)$. By construction, we have
$$
s_i=s_{i_0}=d(b_{i_0},b_{i_1})\p\ldots\p d(b_{i_{t-1}},b_{i_t})\p \Phi^+(a_{i_t}).
$$
We also have $d(b_{i_t},b_k)\leq \Phi^+(a_{i_t},a_k)\leq \Phi^+(a_{i_t})\p \Phi^+(a_k)$ by $(i)$. Altogether,
\begin{align*}
d(b_i,b_k) &\leq d(b_{i_0},b_{i_1})\p\ldots\p d(b_{i_{t-1}},b_{i_t})\p d(b_{i_t},b_k)\\
 &\leq d(b_{i_0},b_{i_1})\p\ldots\p d(b_{i_{t-1}},b_{i_t})\p \Phi^+(a_{i_t})\p \Phi^+(a_k)\\
 &= s_i\p s_k.
\end{align*}

\noindent\textit{Case 2}: $s_k=s_i\p d(b_i,b_k)$ for some $i<k$. 

Then, for any $j<k$, using $(3)$ and induction we have
\begin{enumerate}[\hspace{10pt}$\bullet$]
\item $d(b_j,b_k)\leq d(b_i,b_j)\p d(b_i,b_k)\leq s_i\p s_j\p d(b_i,b_k)=s_j\p s_k$,
\item $s_j\leq s_i\p d(b_i,b_j)\leq s_i\p d(b_i,b_k)\p d(b_j,b_k)=s_k\p d(b_j,b_k)$, and
\item $s_k=s_i\p d(b_i,b_k)\leq s_j\p d(b_j,b_k)$,
\end{enumerate}
and so $(3)$ is satisfied. For $(1)$, we must show $\Phi^-(a_k)<s_i\p d(b_i,b_k)$. As in Case 1, we construct a sequence $i=i_0>i_1>\ldots>i_t$ such that
$$
s_i=d(b_{i_0},b_{i_1})\p\ldots\p d(b_{i_{t-1}},b_{i_t})\p \Phi^+(a_{i_t}).
$$
We want to show
$$
\Phi^-(a_k)<d(b_{i_0},b_{i_1})\p\ldots\p d(b_{i_{t-1}},b_{i_t})\p \Phi^+(a_{i_t})\p d(b_i,b_k).
$$
By the triangle inequality, it suffices to show 
$$
\Phi^-(a_k)<d(b_k,b_{i_t})\p\Phi^+(a_{i_t}).
$$
Since $\Phi^-(a_k,a_{i_t})<d(b_k,b_{i_t})$, this follows from $(ii)$.  
\end{proof}

We can now restate and prove Theorem \ref{prethm2}.

\begin{theorem}\label{QE}
Suppose $\cR$ is a countable distance monoid. The following are equivalent.
\begin{enumerate}[$(i)$]
\item $\TUS{\cR}$ has quantifier elimination.
\item $\TUS{\cR}$ admits extension axioms.
\item For all $s\in R$, the map $x\mapsto x\p s$ is continuous from $R^*$ to $R^*$.
\item For all nonzero $\alpha\in R^*$, if $\alpha$ has no immediate predecessor in $R^*$ then, for all $s\in R$, $\alpha\p s=\sup\{x\p s:x<\alpha\}$.
\end{enumerate}
\end{theorem}   
\begin{proof}
$(iii)\Rightarrow(ii)$: Fix an extension scheme $(\cA,f,\Psi)$. By Proposition \ref{bddRed}, we may assume $(\cA,f,\Psi)$ is standard. By Lemma \ref{metSol}, there is a metric $R$-approximation $\Psi_0$ of $\Dist(\cA^f)$ such that $\Psi_0$ refines $\hat{\Psi}$. We may consider $\Psi_0$ as an $R$-approximation of $\cA^f$, which refines $\Psi$. We define an $R$-approximation $\Phi$ of $\cA^f$ such that $\Phi$ refines $\Psi_0$ and $\cU_\cR\models\epsilon^\Phi_\cA$. By Lemma \ref{extAx}, it suffices to define $\Phi$, refining $\Psi_0$, so that:
\begin{enumerate}[$(1)$]
\item for all $a,b\in A$, $\Phi^+(a,b)\leq\Phi^+(a)\p\Phi^+(b)$,
\item for all $a,b\in A$ and $t\in R$, if $\Phi^-(a,b)<t$ then $\Phi^-(a)<t\p\Phi^+(b)$.
\end{enumerate}

Let $\Phi(a)=\Psi_0(a)$ for all $a\in A$. Given distinct $a,b\in A$, let $\Phi^+(a,b)=\Psi^+_0(a,b)$. Since $\Psi_0$ is metric, we have that for any $a,b\in A$,
$$
\Phi^+(a,b)=\Psi^+_0(a,b)\leq\Psi^+_0(a)\p\Psi^+_0(b)=\Phi^+(a)\p\Phi^+(b),
$$
and so $(1)$ is satisfied. 

Next, we fix $a,b\in A$ and define $\Phi^-(a,b)$ satisfying $(2)$. (Note that, to keep $\Phi$ symmetric, we may need to replace $\Phi^-(a,b)$ with $\max\{\Phi^-(a,b),\Phi^-(b,a)\}$.)

Let $X=\{x\in R^*:\Phi^-(a)<x\p\Phi^+(b)\}$. Then $X$ is the inverse image of $(\Phi^-(a),\omega_R]$ under $x\mapsto x\p\Phi^+(b)$. Since $X$ is closed upward in $R^*$, it follows from $(iii)$ that either $X=R^*$ or $X=(\beta,\omega_R]$ for some $\beta\in R^*$. If $X=R^*$ then $(2)$ is trivially satisfied, and so we may just set $\Phi^-(a,b)=\Psi_0^-(a,b)$. Suppose $X=(\beta,\omega_R]$ for some $\beta\in R^*$. Note that 
$$
\Phi^-(a)<f(a)\leq d_A(a,b)\p f(a)\leq d_A(a,b)\p \Phi^+(b),
$$
and so $d_A(a,b)\in X=(\beta,\omega_R]$. Therefore, by density of $R$, we may choose $u\in R$ such that $\max\{\Psi^-_0(a,b),\beta\}\leq u<d_A(a,b)$, and then set $\Phi^-(a,b)=u$. To verify $(2)$, fix $t\in R$ with $\Phi^-(a,b)<t$. Then $t\in X$, and so $\Phi^-(a)<t\p\Phi^+(b)$, as desired.

$(ii)\Rightarrow(i)$: Fix $M,N\models \TUS{\cR}$ and suppose $A\seq M\cap N$ is a substructure. Fix a quantifier-free formula $\vphi(\xbar,y)$. Suppose that there is $\abar\in A$ and some $b\in M$ such that $M\models\vphi(\abar,b)$. We want to show that there is some $c\in N$ such that $N\models\vphi(\abar,c)$. Without loss of generality, we may assume $\vphi(\xbar,y)$ is a conjunction of atomic and negated atomic formulas. If $b\in \abar$ then we may set $c=b$. So assume $b\not\in\abar$. 

Since $\TMS{R}{\cR}\seq\Th(\cU_\cR)$, we have $\cR^*$-metrics $d_M$ and $d_N$ on $M$ and $N$, respectively (as defined in Theorem \ref{thm:models} and its proof). Let $\cA=(\abar,d_M)$ and define $f:\abar\func R^*$ such that $f(a_i)=d_M(a_i,b)$. Then $f\in E^+_{\cR^*}(\cA)$. Moreover, there is some $R$-approximation $\Psi$ of $\cA^f=(\abar b,d_M)$ such that $\vphi(\xbar,y)$ is equivalent to $C^\Psi_{\abar}(\xbar)\wedge K^\Psi_{\abar}(\xbar,y)$. Since $R$ admits extension axioms, there is an $R$-approximation $\Phi$ of $\cA^f$ such that $\Phi$ refines $\Psi$ and $\cU_\cR\models\epsilon^\Phi_\cA$. Then $N\models C^\Phi_\cA(\abar)$, so there is some $c\in N$ such that $N\models K^\Phi_\cA(\abar,c)$. Since $\Phi$ refines $\Psi$, it follows that $N\models\vphi(\abar,c)$, as desired.

$(i)\Rightarrow(iv)$: Suppose $(iv)$ fails. Fix $s\in R$ and nonzero $\alpha\in R^*$ such that $\alpha$ has no immediate predecessor in $R^*$ and $\sup\{x\p s:x<\alpha\}<\alpha\p s$. By density of $R$, we may fix $t\in R$ such that $\sup\{x\p s:x<\alpha\}\leq t<\alpha\p s$.  By Proposition \ref{saturated2}, there is $M\models \TUS{\cR}$, with $a_1,a_2,b\in M$, such that $d_M(a_1,a_2)=\alpha$, $d_M(a_1,b)=s$, and $d_M(a_2,b)=\alpha\p s$. Define the $\cL_R$-formula
$$
\vphi(x_1,x_2,y):=d(x_1,y)\leq s \wedge d(x_2,y)>t,
$$
and note that $M\models\vphi(a_1,a_2,b)$.

\noindent\textit{Claim:} There is $N\models\Th(\cU_\cR)$, and $a'_1,a'_2\in N$, such that $d_N(a'_1,a'_2)=\alpha$ and $N\models\neg\exists y\vphi(a'_1,a'_2,y)$.

\noindent\textit{Proof}: By compactness it suffices to fix $u,v\in R$, with $u<\alpha\leq v$, and show
$$
\cU_\cR\models \exists x_1 x_2(u<d(x_1,x_2)\leq v\wedge\neg\exists y\vphi(x_1,x_2,y)).
$$
Since $\alpha$ has no immediate predecessor, we may use density of $R$ to fix $w\in R$ such that $u<w<\alpha$. Then $w\p s\leq t$ by choice of $t$. Pick $a'_1,a'_2\in \cU_\cR$ with $d(a'_1,a'_2)=w$. Then $\cU_\cR\models u<d(a'_1,a'_2)\leq v$. If $b'\in\cU_\cR$ is such that $\cU_\cR\models\vphi(a'_1,a'_2,b')$ then
$$
t<d(a'_2,b')\leq d(a'_1,a'_2)\p d(a'_1,b')=w\p d(a'_1,b')\leq w\p s\leq t,
$$
which is a contradiction. So $\cU_\cR\models\neg\exists y\vphi(a'_1,a'_2,y)$.\claim

Let $N$ be as in the claim. Then $M\models \exists y\vphi(a_1,a_2,y)$ and $N\models\neg\exists y\vphi(a'_1,a'_2,y)$. Moreover, $(a_1,a_2)$ and $(a'_1,a'_2)$ both realize $p_\alpha(x_1,x_2)$, and thus have the same quantifier-free type. Therefore $\TUS{\cR}$ does not have quantifier elimination.

$(iv)\Rightarrow(iii)$: Fix $s\in R$ and let $I\seq R^*$ be a sub-basic open interval. Let $X=\{x\in R^*:x\p s\in I\}$. We want to show $X$ is open. Since $\cR^*$ is difference-complete, $x \mapsto x\p s$ is upper semicontinuous (see Remark \ref{rem:USC}), and so we may assume $I=(\beta,\omega_R]$ for some $\beta\in R^*$. If $s>\beta$ then $X=R^*$; so we may assume $s\leq\beta$. Let $\alpha=\sup\{x\in R^*:\beta\geq x\p s\}$. We show $X=(\alpha,\omega_R]$, i.e., $x\leq\alpha$ if and only if $x\p s\leq\beta$. The reverse direction is by definition of $\alpha$. For the forward direction, it suffices to show $\alpha\p s\leq\beta$, and we may clearly assume $\alpha$ is nonzero with no immediate predecessor. By $(iv)$, $\alpha\p s=\sup\{x\p s:x<\alpha\}\leq\beta$.
\end{proof} 

We again recall that, since $\cR^*$ is difference-complete, the maps $x\mapsto x\p s$, for $s\in R$, are always upper semicontinuous. Therefore, quantifier elimination for $\TUS{\cR}$ is precisely equivalent to lower semicontinuity of these maps.

It is also worth observing that one can (carefully) express condition $(iv)$ with first-order properties of $\cR$, although the characterization is more complicated and less intuitive. We leave this as an exercise for the curious reader.

\begin{corollary}
There is a first-order $\LDS$-sentence $\vphi$ such that, for any countable distance monoid $\cR$, $\TUS{\cR}$ has quantifier elimination if and only if $\cR\models\vphi$.
\end{corollary}

In Section \ref{sec:ex}, we will give a number of natural examples, which illustrate that quantifier elimination for $\TUS{\cR}$ holds in many common situations. For now, we give an example where quantifier elimination fails.

\begin{example}\label{noQE}
Let $\cR=(R,+,\leq,0)$, where $R^{>0}=(\Q\cap[2,\infty))\backslash\{3\}$. Let $X$ be the gap cut $(3,\infty)\cap\Q$. Then $g_X+2=5^+$. Therefore, given $\alpha\in R^*$, $\alpha+2>5$ if and only if $\alpha\geq g_X$, and so $x\mapsto x+2$ is not continuous.
\end{example}

Suppose $\cR$ is a distance monoid, and $\TUS{\cR}$ has quantifier elimination. Using classical results in model theory (see e.g. \cite{Mabook}), we can make some immediate observations. For example, $\TUS{\cR}$ is $\aleph_0$-categorical if and only if $R$ is finite; and $\TUS{\cR}$ is \textit{small} (i.e. has a countable saturated model) if and only if $R^*$ is countable. We end this section with an $\forall\exists$-axiomatization of $\TUS{\cR}$.

\begin{definition}
Suppose $\TUS{\cR}$ has quantifier elimination.
\begin{enumerate}
\item Given an extension scheme $(\cA,f,\Psi)$, let $\Phi$ be an $R$-approximation of $\cA^f$ such that $\Phi$ refines $\Psi$ and $\cU_\cR\models\epsilon^{\Phi}_\cA$. Define $\epsilon(\cA,f,\Psi):=\epsilon^{\Phi}_\cA$.
\item Define $\TAX{\cR}=\TMS{R}{\cR}\cup\{\epsilon(\cA,f,\Psi):\text{$(\cA,f,\Psi)$ is an extension scheme}\}$.
\end{enumerate}
\end{definition}

\begin{theorem}
If $\TUS{\cR}$ eliminates quantifiers then it is axiomatized by $\TAX{\cR}$.
\end{theorem}
\begin{proof}
Since we have defined $\TAX{\cR}$ using the extension axioms resulting from quantifier elimination for $\TUS{\cR}$, we may run the same back-and-forth argument as in the proof of Theorem \ref{QE}[$(ii)\Rightarrow (i)]$ to conclude that $\TAX{\cR}$ also has quantifier elimination. Now fix $M,N\models\TAX{\cR}$ and let $a\in M$, $b\in N$ be singletons. Then $\text{qftp}_M(a)=\text{qftp}_N(b)$ (see Remark \ref{rem:cutaxiom}) and so $\text{tp}_M(a)=\text{tp}_N(b)$.  It follows that $M$ and $N$ are elementarily equivalent, and so we have shown that $\TAX{\cR}$ is complete. 
Since $\TAX{\cR}\seq\TUS{\cR}$, this proves the result.
\end{proof}

\section{Examples}\label{sec:ex}

In this section, we consider examples of Urysohn spaces, which arise naturally in the literature, and we verify that they all have quantifier elimination. We continue to use the extension $\cR^*=(R^*,\p,\leq,0)$ as in the previous section.  

\begin{definition}\label{def:exs} Let $\cR=(R,\p,\leq,0)$ be a countable distance monoid.
\begin{enumerate}
\item $\cR$ is \textbf{right-closed} if, for any nonempty subset $X\seq R$, if $\sup X<\sup R$ then $\sup X\in X$.
\item $\cR$ is \textbf{ultrametric} if $r\p s=\max\{r,s\}$ for all $r,s\in R$.

\item $\cR$ is \textbf{group-like} if, for all $r,s\in R$,
\begin{enumerate}[$(i)$]
\item if $s<r$ and $|r\m s|> \inf R^{>0}$ then $r=|r\m s|\p s$;
\item for all $x\in R$, if $|r\m s|<x$ and $r<\sup R$ then $r<x\p s$.
\end{enumerate}
\end{enumerate}
\end{definition}

\begin{remark}$~$
\begin{enumerate}
\item Note that finite distance monoids are right-closed. Urysohn spaces over finite distance sets in $\R^{\geq0}$ are studied in \cite{vThebook} and \cite{Sadiv} from the perspectives of infinitary Ramsey theory and topological dynamics of isometry groups.

\item Suppose $\cR$ is ultrametric. Then $\cU_\cR$ is an ultrametric space with distance set $R$. An important remark is that, in this case, $\TUS{\cR}$ is essentially the theory of infinitely refining equivalence relations, indexed by $(R,\leq)$. These are very common examples, often used in a first course in model theory to exhibit a variety of behavior in the stability spectrum (see e.g. \cite{Babook}). Moreover, ultrametric Urysohn spaces are actively studied in descriptive set theory and topological dynamics of isometry groups (e.g. \cite{GaCh}, \cite{NVTult}).

\item Group-like distance monoids arise naturally in the following way. Fix a countable ordered abelian group $\cG=(G,+,\leq,0)$, and a convex subset $C\seq G^{>0}$, which is either closed under addition or contains a maximal element. Define $\cR=(R,\p,\leq,0)$, where $R=C\cup\{0\}$ and, given $r,s\in R$, $r\p s=\min\{r+s,\sup C\}$. Then $\cR$ is group-like. These examples appear frequently in the literature in the case when $\cG$ is a countable subgroup of $(\R,+,\leq,0)$, and are often included in the general study of Urysohn spaces (see, e.g., \cite{AKL}, \cite{BiMe}, \cite{Sol}). For instance, the rational Urysohn space and rational Urysohn sphere are each examples of $\cU_\cR$ for a group-like distance monoid $\cR$.
\end{enumerate}
\end{remark}

\begin{proposition}
Suppose $\cR$ is a countable distance monoid. If $\cR$ is right-closed, ultrametric, or group-like, then $\Th(\cU_\cR)$ has quantifier elimination.
\end{proposition}
\begin{proof}
We show that, in each case, $\cR$ satisfies Theorem \ref{QE}$(iv)$. Fix nonzero $\alpha\in R^*$, with no immediate predecessor in $R^*$, and $s\in R$. 

If $\cR$ is right-closed then we must have $\alpha=\omega_R$, and so 
$$
\omega_R=\sup\{x:x<\omega_R\}\leq\sup\{x\p s:x<\omega_R\}\leq\omega_R.
$$

If $\cR$ is ultrametric then $\alpha\p s=\max\{\alpha,s\}$. Since $\alpha=\sup\{x:x<\alpha\}$, it follows that $\max\{\alpha, s\}=\sup\{\max\{x,s\}:x<\alpha\}$.

Finally, assume $\cR$ is group-like and suppose, toward a contradiction, that $\sup\{x\p s:x<\alpha\}<\alpha\p s$. By density of $R$, there is $r\in R$ such that $\sup\{x\p s:x<\alpha\}\leq r<\alpha\p s$. Note that $s\leq r<\sup R$. If $|r\m s|=\alpha$ then $|r\m s|> \inf R^{>0}$, since $\inf R^{>0}$ has an immediate predecessor in $R^*$, namely $0$. But then $r=|r\m s|\p s=\alpha\p s$, which contradicts the choice of $r$. Therefore $|r\m s|<\alpha$. By density of $R$, there is $x\in R$ such that $|r\m s|<x<\alpha$. But then $x<\alpha$ and $r<x\p s$, which contradicts the choice of $r$.
\end{proof}

We end this section with a discussion of a particular family of generalized Urysohn spaces, which have been used in previous work to obtain exotic behavior in model theory. First, however, we give a more explicit axiomatization of $\TUS{\cR}$ in the case that $\cR$ is finite.

Note that if $\cR$ is a finite distance monoid, then we have $\cR^*=\cR$. In this case, given $r\in R$ with $r>0$, we let $r^-$ denote the immediate predecessor of $r$. 

\begin{definition}
Suppose $\cR$ is a finite distance monoid. Given a finite $\cR$-metric space $\cA$, the \textbf{canonical $R$-approximation of $\cA$} is the function $\Phi_\cA:A\times A\func \Int(R)$ such that $\Phi_\cA(a,b)=(d_A(a,b)^-,d_A(a,b)]$ for $a\neq b$. If $f\in E^+_{\cR}(\cA)$, we let $\epsilon(\cA,f)$ denote the extension axiom $\epsilon^{\Phi_{\cA^f}}_\cA$.
\end{definition}

If $\cR$ is a finite distance monoid, and $\cA$ is a finite $\cR$-metric space, then $\Phi_\cA$ refines any $R$-approximation of $\cA$. Moreover,  if $f\in E^+_\cR(\cA)$ then $\cU_\cR\models\epsilon(\cA,f)$ (this can be shown directly or via Lemma \ref{extAx}). Altogether, we may define the axiomatization $\TAX{\cR}$ so that, given an extension scheme $(\cA,f,\Psi)$, we set $\epsilon(A,f,\Psi)=\epsilon(\cA,f)$. In particular, the extension axiom for $(\cA,f,\Psi)$ depends only on $\cA$ and $f$. This axiomatization also agrees with the usual $\forall\exists$-axiomatization of $\aleph_0$-categorical \Fraisse\ limits.

We now turn to a specific family of examples. Given $n>0$, set $R_n=\{0,1,2,\ldots,n\}$ and $S_n=\{0,\frac{1}{n},\frac{2}{n},\ldots,1\}$, and let $+_n$ denote addition truncated at $n$. Define the distance monoids $\cR_n=(R_n,+_n\leq,0)$, $\cS_n=(S_n,+_1,\leq,0)$, and $\cQ_1=(\Q\cap[0,1],+_1,\leq,0)$. Note that $\cS_n$ is a submonoid of $\cQ_1$.

In \cite{CaWa}, Casanovas and Wagner construct $T_n$, the theory of the \textit{free $n\uth$ root of the complete graph}, for $n>0$. In particular, $T_1$ is the theory of an infinite complete graph; and $T_2$ is the theory of the random graph. The reader familiar with their work will recognize that, for general $n>0$, $T_n$ is precisely $\TAX{\cR_n}$ (using the canonical extension axioms). In order to form a directed system of first-order theories, Casanovas and Wagner then replace $\cR_n$ with $\cS_n$ and define $T_\infty=\bigcup_{n>0}\TAX{\cS_n}$. We now verify that $T_\infty$ axiomatizes $\TUS{\cQ_1}$, the theory of the rational Urysohn sphere. To do this, we prove the following proposition.

\begin{proposition}
$\TUS{\cQ_1}=\bigcup_{n<\omega}\TUS{\cS_n}$. 
\end{proposition}
\begin{proof}
We first fix $n>0$ and show $\TAX{\cS_n}\seq\TUS{\cQ_1}$. Note that $\Delta(\cU_{\cQ_1})\subsetsim\Delta(S_n,\cQ_1)$ (see Example \ref{ex:metDense}), and so $\TMS{S_n}{\cQ_1}\seq \TUS{\cQ_1}$ by Proposition \ref{convAx}. Therefore, we must fix a finite $\cS_n$-metric space $\cA$ and $f\in E^+_{\cS_n}(\cA)$, and show $\cU_{\cQ_1}\models\epsilon(\cA,f)$. In particular, we use Lemma \ref{extAx}. Let $\Phi$ be the canonical $S_n$-approximation of $\cA^f$. Given distinct $a,b\in A$, we clearly have $\Phi^+(a,b)\leq\Phi^+(a)+_1\Phi^+(b)$. Next, fix $a,b\in A$ and $s\in \Q\cap[0,1]$ with $\Phi^-(a,b)<s$. Let $d_A(a,b)=\frac{k}{n}$, $f(a)=\frac{i}{n}$, and $f(b)=\frac{j}{n}$, where $0<i,j,k\leq n$. Then we have $s>\frac{k-1}{n}$, and we want to show $\frac{i-1}{n}<s+_1\frac{j}{n}$.
We obviously have $\frac{i-1}{n}<1$, so it suffices to show $i-1<ns+j$. Since $f\in E^+_{\cS_n}(\cA)$, we have $i\leq k+j$, and so $i-1\leq k-1+j<ns+j$, as desired.   

We now have $\bigcup_{n<\omega}\TUS{\cS_n}\seq \Th(\cU_{\cQ_1})$. Since $\TUS{\cS_n}$ is a complete $\cL_{S_n}$-theory for all $n>0$, and $\cL_{\Q\cap[0,1]}=\bigcup_{n>0}\cL_{S_n}$, the desired result follows.
\end{proof}

Casanovas and Wagner remark that saturated models of $T_\infty$ could be treated as metric spaces with ``nonstandard" distances in $(\Q\cap[0,1])^*$, but it is not observed that $T_\infty$ is the theory of such a classical structure. The main result of \cite{CaWa} is that $T_\infty$ does not eliminate hyperimaginaries. In particular, let $E(x,y)=\{d(x,y)\leq r:r\in\Q\cap(0,1]\}$ be the type-definable equivalence relation describing infinitesimal distance. Then the $E$-equivalence class of any singleton (in some sufficiently saturated model) is a non-eliminable hyperimaginary. In Section 7 of the sequel to this paper \cite{CoDM2}, we generalize their methods in the setting of an arbitrary countable distance monoid $\cR$ and, assuming quantifier elimination, obtain necessary conditions for elimination of hyperimaginaries in $\Th(\cU_\cR)$. Along the way, we also characterize weak elimination of imaginaries for $\Th(\cU_\cR)$. 

In \cite{CaWa}, Casanovas and Wagner show that $\TUS{\cQ_1}$ is non-simple and without the strict order property. In \cite{CoTe}, it is shown that the theory of the complete Urysohn sphere in continuous logic has the strong order property, but not the fully finitary strong order property. In \cite{CoDM2}, we refine and extend these methods to characterize the neostability theoretic behavior of $\Th(\cU_\cR)$, for any countable distance monoid $\cR$ such that $\Th(\cU_\cR)$ has quantifier elimination.

\section*{Acknowledgments} 
This work was done while under the supervision of my thesis advisor, David Marker. I would also like to thank John Baldwin, Ward Henson, and Christian Rosendal for helpful conversations, as well as Julien Melleray for first introducing me to the four-values condition and the work in \cite{DLPS}. Finally, I thank the anonymous referee for improving the exposition.

\bibliography{/Users/gabrielconant/Desktop/Math/BibTex/biblio}

\def\cprime{$'$}
\providecommand{\bysame}{\leavevmode\hbox to3em{\hrulefill}\thinspace}
\providecommand{\MR}{\relax\ifhmode\unskip\space\fi MR }
\providecommand{\MRhref}[2]{%
  \href{http://www.ams.org/mathscinet-getitem?mr=#1}{#2}
}
\providecommand{\href}[2]{#2}
\begin{thebibliography}{10}

\bibitem{AlTr}
C.~Alsina and E.~Trillas, \emph{On natural metrics}, Stochastica \textbf{2}
  (1977), no.~3, 15--22. \MR{562428 (81d:08001)}

\bibitem{AKL}
Omer Angel, Alexander~S. Kechris, and Russell Lyons, \emph{Random orderings and
  unique ergodicity of automorphism groups}, J. Eur. Math. Soc. (JEMS)
  \textbf{16} (2014), no.~10, 2059--2095. \MR{3274785}

\bibitem{Babook}
John~T. Baldwin, \emph{Fundamentals of stability theory}, Perspectives in
  Mathematical Logic, Springer-Verlag, Berlin, 1988. \MR{918762 (89k:03002)}

\bibitem{BBHU}
Ita{\"{\i}} Ben~Yaacov, Alexander Berenstein, C.~Ward Henson, and Alexander
  Usvyatsov, \emph{Model theory for metric structures}, Model theory with
  applications to algebra and analysis. {V}ol. 2, London Math. Soc. Lecture
  Note Ser., vol. 350, Cambridge Univ. Press, Cambridge, 2008, pp.~315--427.
  \MR{2436146 (2009j:03061)}

\bibitem{BiMe}
Do{\v{g}}an Bilge and Julien Melleray, \emph{Elements of finite order in
  automorphism groups of homogeneous structures}, Contrib. Discrete Math.
  \textbf{8} (2013), no.~2, 88--119. \MR{3251962}

\bibitem{Bour}
Nicolas Bourbaki, \emph{Algebra. {I}. {C}hapters 1--3}, Elements of Mathematics
  (Berlin), Springer-Verlag, Berlin, 1989, Translated from the French, Reprint
  of the 1974 edition. \MR{979982 (90d:00002)}

\bibitem{CaWa}
Enrique Casanovas and Frank~O. Wagner, \emph{The free roots of the complete
  graph}, Proc. Amer. Math. Soc. \textbf{132} (2004), no.~5, 1543--1548
  (electronic). \MR{2053363 (2005f:03046)}

\bibitem{Cliff}
A.~H. Clifford, \emph{Totally ordered commutative semigroups}, Bull. Amer.
  Math. Soc. \textbf{64} (1958), 305--316. \MR{0100641 (20 \#7070)}

\bibitem{CoDM2}
Gabriel Conant, \emph{Neostability in countable homogeneous metric spaces},
  arXiv:1504.02427 [math.LO], 2015.

\bibitem{CoTe}
Gabriel Conant and Caroline Terry, \emph{Model theoretic properties of the
  {U}rysohn sphere}, Ann. Pure Appl. Logic \textbf{167} (2016), no.~1, 49--72.
  \MR{3413493}

\bibitem{DLPS}
Christian Delhomm{\'e}, Claude Laflamme, Maurice Pouzet, and Norbert Sauer,
  \emph{Divisibility of countable metric spaces}, European J. Combin.
  \textbf{28} (2007), no.~6, 1746--1769. \MR{2339500 (2009b:54030)}

\bibitem{GaCh}
Su~Gao and Chuang Shao, \emph{Polish ultrametric {U}rysohn spaces and their
  isometry groups}, Topology Appl. \textbf{158} (2011), no.~3, 492--508.
  \MR{2754373 (2011m:54026)}

\bibitem{Hobook}
Wilfrid Hodges, \emph{Model theory}, Encyclopedia of Mathematics and its
  Applications, vol.~42, Cambridge University Press, Cambridge, 1993.
  \MR{1221741 (94e:03002)}

\bibitem{Kat}
M.~Kat{\v{e}}tov, \emph{On universal metric spaces}, General topology and its
  relations to modern analysis and algebra, {VI} ({P}rague, 1986), Res. Exp.
  Math., vol.~16, Heldermann, Berlin, 1988, pp.~323--330. \MR{952617
  (89k:54066)}

\bibitem{MacN}
H.~M. MacNeille, \emph{Partially ordered sets}, Trans. Amer. Math. Soc.
  \textbf{42} (1937), no.~3, 416--460. \MR{1501929}

\bibitem{Mabook}
David Marker, \emph{Model theory}, Graduate Texts in Mathematics, vol. 217,
  Springer-Verlag, New York, 2002. \MR{1924282 (2003e:03060)}

\bibitem{MeUS}
Julien Melleray, \emph{Some geometric and dynamical properties of the {U}rysohn
  space}, Topology Appl. \textbf{155} (2008), no.~14, 1531--1560. \MR{2435148
  (2009k:54054)}

\bibitem{MoSh}
John~W. Morgan and Peter~B. Shalen, \emph{Valuations, trees, and degenerations
  of hyperbolic structures. {I}}, Ann. of Math. (2) \textbf{120} (1984), no.~3,
  401--476. \MR{769158 (86f:57011)}

\bibitem{Nar}
Louis Narens, \emph{Field embeddings of generalized metric spaces}, Victoria
  {S}ymposium on {N}onstandard {A}nalysis ({U}niv. {V}ictoria, {V}ictoria,
  {B}.{C}., 1972), Springer, Berlin, 1974, pp.~155--170. Lecture Notes in
  Math., Vol. 369. \MR{0645171 (58 \#31037)}

\bibitem{NVTult}
Lionel Nguyen Van~Th{\'e}, \emph{Ramsey degrees of finite ultrametric spaces,
  ultrametric {U}rysohn spaces and dynamics of their isometry groups}, European
  J. Combin. \textbf{30} (2009), no.~4, 934--945. \MR{2504653 (2010j:05418)}

\bibitem{vThebook}
\bysame, \emph{Structural {R}amsey theory of metric spaces and topological
  dynamics of isometry groups}, Mem. Amer. Math. Soc. \textbf{206} (2010),
  no.~968, x+140. \MR{2667917 (2011k:03096)}

\bibitem{Sadiv}
N.~W. Sauer, \emph{Vertex partitions of metric spaces with finite distance
  sets}, Discrete Math. \textbf{312} (2012), no.~1, 119--128. \MR{2852515
  (2012m:54045)}

\bibitem{Sa13}
\bysame, \emph{Distance sets of {U}rysohn metric spaces}, Canad. J. Math.
  \textbf{65} (2013), no.~1, 222--240. \MR{3004464}

\bibitem{Sa13b}
\bysame, \emph{Oscillation of {U}rysohn type spaces}, Asymptotic geometric
  analysis, Fields Inst. Commun., vol.~68, Springer, New York, 2013,
  pp.~247--270. \MR{3076154}

\bibitem{Sol}
S{\l}awomir Solecki, \emph{Extending partial isometries}, Israel J. Math.
  \textbf{150} (2005), 315--331. \MR{2255813 (2007h:54034)}

\bibitem{TeZibdd}
Katrin Tent and Martin Ziegler, \emph{The isometry group of the bounded
  {U}rysohn space is simple}, Bull. Lond. Math. Soc. \textbf{45} (2013), no.~5,
  1026--1030. \MR{3104993}

\bibitem{TeZiSIR}
\bysame, \emph{On the isometry group of the {U}rysohn space}, J. Lond. Math.
  Soc. (2) \textbf{87} (2013), no.~1, 289--303. \MR{3022717}

\bibitem{Ury}
Paul Urysohn, \emph{Sur un espace m{\'e}trique universel}, C. R. Acad. Sci.
  Paris \textbf{180} (1925), 803--806.

\bibitem{Ury2}
\bysame, \emph{Sur un espace m{\'e}trique universel}, Bull. Sci. Math
  \textbf{51} (1927), 43--64, 74--96.

\end{thebibliography}
\bibliographystyle{amsplain}

\end{document}